\documentclass[12pt]{amsart}
\usepackage{amsmath, braket, graphicx, amscd, amssymb, amsthm,mathrsfs}
\usepackage[margin=1.15in]{geometry}
\usepackage{extarrows}
\usepackage{adjustbox}
\usepackage[all,cmtip]{xy}
\usepackage{tikz}
\usepackage{etoolbox}%to avoid underfull hbox warnings in the bibliography
\apptocmd{\sloppy}{\hbadness 10000\relax}{}{}
\usetikzlibrary{trees}
\definecolor{hot}{RGB}{65,105,225}
\usepackage[pagebackref=true,colorlinks=true, linkcolor=hot ,  citecolor=hot, urlcolor=hot]{hyperref}%make all the references and links clickable
\usepackage[alphabetic]{amsrefs}

\numberwithin{equation}{subsection} %numbering of equations is ordered up to subsections. It can be changed to up to sections as well.

\theoremstyle{plain}
\newtheorem{theorem}{Theorem}[section]
\newtheorem{proposition}[theorem]{Proposition}
\newtheorem{prop}[theorem]{Proposition}
\newtheorem{lemma}[theorem]{Lemma}
\newtheorem{corollary}[theorem]{Corollary}
\newtheorem{cor}[theorem]{Corollary}
\theoremstyle{definition}
\newtheorem{definition}[theorem]{Definition}
\newtheorem{notation}[theorem]{Notation}
\newtheorem{defn}[theorem]{Definition}

\newtheorem{example}[theorem]{Example}
\newtheorem{rmk}[theorem]{Remark}
\theoremstyle{remark}

\newcommand{\bigslant}[2]{{\raisebox{.2em}{$#1$}\left/\raisebox{-.2em}{$#2$}\right.}} %making quotients with a bigger diagonal slash
\newcommand{\ubul}{{\,\begin{picture}(-1,1)(-1,-3)\circle*{2}\end{picture}\ }}

\newcommand{\ra}{\rightarrow}

\DeclareMathOperator{\spec}{Spec}
\DeclareMathOperator{\Hom}{Hom}

\DeclareMathOperator{\End}{End}

\DeclareMathOperator{\MC}{MC}
\DeclareMathOperator{\Def}{Def}

\DeclareMathOperator{\art}{\texttt{\emph{Art}}}

\DeclareMathOperator{\id}{id}

\def\cM{\mathcal{M}}
\def\cV{\mathcal{V}}
\def\cO{\mathcal{O}}
\def\bC{\mathbb{C}}
\def\bA{\mathbb{A}}
\def\bL{\mathbb{L}}
\def\bP{\mathbb{P}}

\def\bN{\mathbb{N}}
\def\bK{\mathbb{K}}

\def\bZ{\mathbb{Z}}
\def\be{\begin{equation}}
\def\ee{\end{equation}}

\newcommand{\mb}{\mathbf{M}_{\textup{B}}}

\def\Linf{L_\infty}
\def\mA{\mathfrak{m}_A}
\def\xra{\xrightarrow}
\def\bone{\mathbf{1}}
\def\tart{\texttt{Art}}

\def\Ainf{A_\infty}

\def\bdf{\begin{definition}}
\def\edf{\end{definition}}
\def\cL{\mathcal{L}}
\def\sL{\mathscr{L}}

\begin{document}

\title{%\MakeUppercase?
{L-infinity pairs and applications to singularities}}
\author{Nero Budur} 
\address{KU Leuven, Celestijnenlaan 200B, B-3001 Leuven, Belgium} 
\email{nero.budur@kuleuven.be} 
\email{marcel.rubio@kuleuven.be}

\author{Marcel Rubi\'o}

%\date{ }

\begin{abstract} Over a field of characteristic zero, every deformation problem with cohomology constraints is controlled by a pair consisting of a differential graded Lie algebra together with a  module. Unfortunately, these pairs are usually infinite-dimensional. We show that every deformation problem with cohomology constraints is controlled by a typically finite-dimensional $L_\infty$ pair. As a first application, we show that for complex algebraic varieties with no weight-zero 1-cohomology classes, the components of the cohomology jump loci of rank one local systems containing the constant sheaf are tori.  This imposes restrictions on the fundamental groups. The same holds for links and Milnor fibers.
\end{abstract}

\maketitle

\setcounter{tocdepth}{1}

\tableofcontents

\section{Introduction}

Consider a geometric object endowed with a cohomology theory. How can one describe all its deformations constrained by the condition that the degree $i$ cohomology vector space has dimension $\ge k$?

The  main goal of this article is to provide an answer to this question over fields of characteristic zero. In order to motivate the general theory, we start with a description of the applications we give.

\subsection{Singular varieties}\label{subCJr1}

Let $X$ be a connected topological space having the homotopy type of a finite CW-complex.
The space $\mb (X)$   of all rank one $\bC$-local systems on $X$ is identified with the group $\Hom (\pi_1(X),\bC^*)$ of rank one representations of the fundamental group $\pi_1(X)$ based at a fixed point of $X$. $\mb(X)$ is an algebraic group, the product of a finite abelian group with the complex affine torus $(\bC^*)^b$, where $b$ is the first Betti number of $X$.

Define the {\it cohomology jump loci}
$$
\Sigma^i_k(X)=\{\sL\in \mb (X)\mid \dim H^i(X,\sL)\ge k\}. 
$$ 
 These loci can be endowed with a natural structure of closed subschemes of $\mb (X)$. The cohomology jump loci are homotopy invariants of the topological space $X$. Moreover,  $\Sigma^1_k(X)$ depends only on $\pi_1(X)$ and $k$. Hence one can define $\Sigma^1_k(G)$ for any finitely presented group $G$. 
 
 There is a Murphy's Law: every affine scheme of finite type over $\bC$ can be  $\Sigma^1_k(G)$ (respectively, $\Sigma^i_k(X)$) for some finitely presented group $G$ (respectively, for some finite CW-complex $X$).

The cohomology jump loci are known to be special in certain  cases. By a {\it  subtorus} of $\mb(X)$ we mean an algebraic subgroup $(\bC^*)^p\subset \mb (X)$. The irreducible components of (the underlying reduced algebraic set of) $\Sigma^i_k(X)$ were shown to be torsion-translated subtori if $X$ is a smooth quasi-projective complex algebraic variety  \cite{BW-qproj}, a compact K\"ahler manifold \cite{W-cKm}, the complement in a small ball of a complex analytic set \cite{BW-small-ball}, and (without the ``torsion'' part) if $X$ is a quasi-compact K\"ahler manifold \cite{BW17}.  This problem has a long history, starting with Beauville, Catanese, Green-Lazarsfeld, see \cite{BW-survey}. 

Beyond the smooth case, the torsion-translated subtorus property was shown for $\Sigma^1_k(X)=\Sigma^1_k(\pi_1(X))$ for normal varieties $X$ by Arapura, Dimca, and Hain \cite{ADH}. Also, all $\Sigma^i_k(X)$ satisfy the translated subtorus property for any complex algebraic variety $X$ admitting a morphism to $f:X\ra Y$ to smooth algebraic variety $Y$ with an injection $f_*:H_1(X,\bZ)\hookrightarrow H_1(Y,\bZ)$. This follows by applying \cite{BW-abs}*{Theorem 10.1.1} to $Rf_*\bC_X$. 

A natural question is then  how singular can a complex variety $X$ be in order for the cohomology jump loci  to be special? 

There is a Murphy's Law in this direction too. By  Simpson \cite{Simp-sing}, Kapovich-Koll\'ar \cite{KaKo}, and, in the form  stated here, by Kapovich \cite{Kapo}, any finitely presented group is the fundamental group of an irreducible projective surface with at most normal crossings and Whitney umbrellas as singularities. Hence, if these singularities are allowed, any affine scheme of finite type can occur as a $\Sigma^1_k(X)$. 

We prove the following:

\begin{theorem}\label{corW}
Let $X$ be a connected complex algebraic variety, possibly reducible. If $W_0H^1(X,\bC)=0$, where $W$ is the weight filtration, then each irreducible component of the algebraic set $\Sigma^i_k(X)$ containing the constant sheaf is a subtorus of $\mb(X)$.
\end{theorem}

It is known that $W_0H^1(X,\bC)=0$ if $X$ is a normal variety, or more generally, if $X$ is unibranch. The vanishing $W_0H^1(X,\bC)=0$ also holds if $X$ is a variety which admits a morphism to a smooth variety $Y$ with an injection $H_1(X,\bZ)\hookrightarrow H_1(Y,\bZ)$, by functoriality of mixed Hodge structures.

The theorem suggests that the dimension of $W_0H^1(X,\bC)$ should be a topological invariant for irreducible varieties. If $X$ is compact, this was known by \cite{We}. After we posed this question to M. Saito, he has  promptly provided a proof of the topological invariance beyond the compact case for all equidimensional complex varieties \cite{MS}, cf. Remark \ref{rmkW}.

The proof of the above theorem also works for other spaces behaving like algebraic varieties, such as links and Milnor fibers of singularities. Let $X$ be a complex projective variety, $Z$ and $Z'$ closed subschemes, $Y=Z\cup Z'$, and assume that the singular locus of $X$ is contained in $Y$. The {\it link of $Z$ in $X$ with $Y$ removed} is the complement $\cL=\cL(X,Y,Z):=T-Y$ for a nice neighborhood $T$ of $Z$ in $X$, see \cite{DH}. If $Z=\{x\}$ is an isolated singularity of $X$ and $Z'$ is empty, then $\cL$ is the usual link of the singularity $(X,x)$. The cohomology groups and the rational homotopy type of $\cL$ are endowed with rational mixed Hodge structures by Durfee-Hain \cite{DH}.

\begin{theorem}\label{thrmLink}  If a connected component $\cL'$ of the link $\cL$ satisfies $W_0H^1(\cL',\bC)=0$, then each irreducible component of the algebraic set $\Sigma^i_k(\cL')$ containing the constant sheaf is a subtorus of $\mb(\cL')$.
\end{theorem}

In particular, one has restrictions on fundamental groups of such links. In contrast, there is a Murphy's Law for links too: every finitely presented group is the fundamental group of a link of an isolated complex singularity of dimension 3, by Kapovich-Koll\'ar \cite{KaKo}. Unlike for varieties, $W_0H^1(\cL,\bC)=0$ is not a topological property of a connected link, by Steenbrink-Stevens \cite{SS}*{\S 3}.

For a germ of a holomorphic function $f:(\bC^n,0)\ra (\bC,0)$, let $F$ denote the Milnor fiber. A mixed Hodge structure on the cohomology of $F$ has been constructed by Steenbrink \cite{St-F}, Navarro \cite{Na}, and Saito \cite{Sa-F}.  

\begin{theorem}\label{thrmMFiber}  If a connected component $F'$ of the Milnor fiber $F$ satisfies $W_0H^1(F',\bC)=0$, then each irreducible component of the algebraic set $\Sigma^i_k(F')$ containing the constant sheaf is a subtorus of $\mb(F')$.
\end{theorem}

The condition  $W_0H^1(X)=0$ implies also Morgan-type obstructions on fundamental groups for varieties, see Proposition \ref{rmkADH}, since the proof of \cite{ADH}*{Theorem 1.2} extends beyond the normal case. In fact, we were informed by R. Hain that this also holds for links and Milnor fibers, see Proposition \ref{propMoLi}. These obstructions do not depend on the deformation theory developed below, we only include them for completion.

\subsection{Strategy}
Theorem \ref{corW} is a consequence of three fundamental results: deformation theory with cohomology constraints, an exponential Ax-Lindemann theorem, and strictness with respect to the weight filtration of the higher multiplication maps on the cohomology of algebraic varieties. 

For a connected topological space $X$, let $\Omega_{{\rm DR}}(X)$ be Sullivan's commutative differential graded algebra (abbreviated {\it cdga} from now) of piecewise smooth $\bC$-forms on $X$. It is known that $\Omega_{{\rm DR}}(X)$ governs {the} structure of  $\Sigma^i_k(X)$ around the constant sheaf for all integers $i$ and $k$ according to \cites{DPS, DimPap14, BW}; the precise meaning of this statement will be explained in \ref{subDP}.

If $\Omega_{{\rm DR}}(X)$ is equivalent to a finite-dimensional cdga, then  every component  of $\Sigma^i_k(X)$ containing the constant sheaf is a subtorus. The reason is that in this case one can construct closed subschemes $R^{\,i}_k(X)$ in $H^1(X,\bC)$, called {\it resonance varieties}, whose image under the exponential map recovers the components of $\Sigma^i_k(X)$ containing the constant sheaf. Then one can apply the following exponential Ax-Lindemann theorem from \cite{BW17}*{Corollary 2.2}:

\begin{theorem}\label{propAx}
Suppose $(W, 0)$ and $(V, 1)$ are analytic germs of two algebraic sets in $\bC^n$ and $(\bC^*)^n$, respectively. If the exponential map $\exp: \bC^n\to (\bC^*)^n$ induces an isomorphism between $(W, 0)$ and $(V, 1)$, then $(V, 1)$ is the germ of a finite union of subtori. 
\end{theorem}

It is known that $\Omega_{{\rm DR}}(X)$ is equivalent to a finite-dimensional cdga if $X$ is a quasi-compact K\"ahler manifold \cite{Mo}. Beyond this case, we need a new strategy. The strategy for the proof of Theorem \ref{corW} is to force finiteness. There will be a price to pay for doing so. Then one has to improve the deformation theory to make sure the price paid for finiteness is affordable.

Forcing finiteness for the controlling cdga is not a new idea. We will use the minimality theorem of Kadeishvili which implies that $\Omega_{{\rm DR}}(X)$ is always $A_\infty$ equivalent to the cohomology $H(X)$ endowed with a  structure of  $A_\infty$ algebra, canonical up to $A_\infty$ isomorphism. The price paid is having to work with $A_\infty$ algebras instead of cdga's.

An {$A_\infty$ algebra} is a graded vector space $A$ together with graded linear maps  $$\mu_n:A^{\otimes n}\ra A$$ of degree $2-n$ for each $n\ge 1$, satisfying associativity up to homotopy. Every dga is an $A_\infty$ algebra with $\mu_1$ the differential, $\mu_2$ the multiplication, and $\mu_{>2}=0$. On the cohomology $H(X)$, $\mu_1=0$, $\mu_2$ is the cup product, and the higher $\mu_n$ are related to the higher Massey products.

By Cirici-Horel \cite{CH}, for every complex variety $X$, Sullivan's de Rham cdga is equivalent to a cdga admitting an extra grading compatible with the differential and the multiplication, and inducing the weight filtration on $H(X)$, see Theorem \ref{thrmCH}. This was first proven by Morgan \cite{Mo} in the case of smooth varieties and extended by Cirici-Guill\'en \cite{CG} to possibly singular nilpotent varieties. The approach of \cite{CH} allows to remove the nilpotency conditions in the singular case. The result of Cirici-Horel holds for links and Milnor fibers of singularities as well, cf. Remark \ref{rmkLinkCH}.

Using Theorem \ref{thrmCH}, we prove that the $A_\infty$ products on $H(X)$ are strict with respect to the weight filtration, see Theorem \ref{thrmStrict2}.

To apply all this, an improved version of deformation theory with cohomology constraints is  needed, where the $A_\infty$ algebra $H(X)$ replaces  the cdga $\Omega_{{\rm DR}}(X)$. We will describe this in the next subsections. For now, to finish describing the proof of Theorem \ref{corW}, we  mention that the condition $W_0H^1(X)=0$ guarantees finiteness of the relevant equations describing locally $\Sigma^i_k(X)$ around the constant sheaf. That is,  resonance varieties can be again constructed as honest schemes of finite type over $\bC$, and the exponential Ax-Lindemann theorem applies  to conclude Theorem \ref{corW}.

In practice, to apply deformation theory we work with differential graded Lie algebras (abbreviated {\it dgla}'s) and $L_\infty$ algebras, and with modules over these. We also work not necessarily only around the constant sheaf. The structure of the loci $\Sigma^i_k(X)$ around a local system $\sL$ is governed by the pair $$(\Omega_{{\rm DR}}(X),\Omega_{{\rm DR}}(\sL))$$ by \cite{BW}, as we explain in subsection \ref{subDP}, where $\Omega_{{\rm DR}}(\sL)$ is the dgl module over $\Omega_{{\rm DR}}(X)$  obtained from the forms with values in the rank one local system $\sL$.

If the dgl pair $(\Omega_{{\rm DR}}(X),\Omega_{{\rm DR}}(\sL))$ is equivalent to a finite-dimensional dgl pair, then the translated-subtorus property  holds for every component  of $\Sigma^i_k(X)$ containing $\sL$,  by \cite{BW17}*{Theorem 1.3}. We improve here on this  result as follows. There is an $L_\infty$ pair structure unique up to $\Linf$ pair isomorphism on $(H(X,\bC),H(X,\sL))$  making it equivalent as an $L_\infty$ pair with the dgl pair $(\Omega_{{\rm DR}}(X),\Omega_{{\rm DR}}(\sL))$, see Theorem \ref{thmMainLDef}.   The $\Linf$ algebra structure on $H(X,\bC)$ is trivial, but the $\Linf$ module structure of $H(X,\sL)$ is not, even when $\sL$ is the constant sheaf; we call this {\it a canonical $\Linf$ module structure}. When $\sL$ is the constant sheaf, this is essentially an $A_\infty$ algebra structure on $H(X,\bC)$ making it $A_\infty$ equivalent to the cdga $\Omega_{{\rm DR}}(X)$; we call it  {\it a canonical $A_\infty$ algebra structure}.

\begin{theorem}
\label{thrmLoc} Let $X$ be a connected topological space, homotopy equivalent to a finite CW-complex.

(1) If there exists $n_0$ and a canonical  $A_\infty$ algebra structure $m=(m_n)_{n\ge 2}$ on  $H(X,\bC)$ with $$m_n(\omega,\ldots,\omega,\eta)=0$$  for all $n>n_0$, $\omega\in H^1(X,\bC)$, $\eta\in H(X,\bC)$, then each irreducible component of the algebraic set $\Sigma^i_k(X)$ containing the trivial local system is a subtorus of $\mb(X)$.

(2) Let $\sL$ be a rank one local system on $X$. If there exists $n_0$ 
and a canonical  $L_\infty$ module structure $m=(m_n)_{n\ge 2}$  on  $H(X,\sL)$ over  $H(X,\bC)$,  with $$m_n(\omega,\ldots,\omega,{\eta})=0$$  for all $n>n_0$, $\omega\in H^1(X,\bC)$, ${\eta}\in H^\ubul(X,\sL)$, then every irreducible component of the algebraic set $\Sigma^i_k(X)$ passing through $\sL$ is a translated subtorus of $\mb(X)$.
\end{theorem}

Part (1) is satisfied, and therefore implies Theorem \ref{corW}, if $W_0H^1(X)=0$ for an algebraic variety, by strictness of the higher multiplication maps on $H(X)$.

\subsection{Deformation problems with cohomology constraints}\label{subDP}  The main technical part of the article is a general solution for all deformation problems with cohomology constraints over a field of characteristic zero. We describe now a personal view of the state-of-the-art and introduce some notation. In the next subsection, we state our results.

A {\it deformation problem} means: describe all the deformations keeping a certain structure of an object up to isomorphisms. In the presence of a moduli space $\cM$, say as scheme whose closed points parametrize the isomorphism classes of the objects to be deformed, the deformation problem for an object $\rho\in\cM$ is equivalent to describing the formal germ $\cM_{\rho}$ of $\cM$ at $\rho$.

 A principle of Deligne \cite{Del86} says that, over a field of characteristic zero, every deformation problem is controlled by a dgla, and two equivalent dgla's describe the same deformation theory.  This means that, as functors on Artinian local algebras, $\cM_{\rho}$ is naturally isomorphic to the deformation functor $\Def (C)$ attached to some dgla $C$, defined as the Maurer-Cartan elements in $H^1C$ modulo the gauge action of $H^0C$, and if $C'$ is a dgla equivalent to $C$, then $\Def(C)\simeq\Def(C')$. 

This principle has been illustrated many times, beginning with the work of Goldman-Millson \cite{GM88}, where the bases of the deformation theory in terms of dgla's were built upon previous work of Schlessinger, Stasheff. 

The passage from functors to actual equations is however only partially answered by this approach. In order for $\Def(C)$ to provide useful equations describing $\cM_{\rho}$, one would need at least that $C$ is finite-dimensional. Typically this is not the case, e.g. $\Omega_{{\rm DR}}(X)$.

The approach via dgla's to deformation theory is particularly successful in the situation when $C$ is equivalent to its cohomology dgla $HC$, that is, when $C$ is a {\it  formal} dgla. Typically, $HC$ is finite-dimensional. In addition, the vanishing of the differential in $HC$ translates into a beautiful answer for the deformation problem in terms of the linear algebra of the Lie bracket on $HC$: the space $\cM$ has quadratic singularities at $\rho$ up to gauge, see \cite{GM88} and \cite{Man04} for more on this subject.
 
Beyond formality, the  approach via dgla's to deformation problems lacks the necessary power to provide equations. Unless one is able to find at least an equivalent finite-dimensional dgla, however this is typically a difficult process, and in no way guaranteed.

The final answer to providing equations for $\cM_{\rho}$ even beyond formality is given by $L_\infty$ algebras. These come with a possibly-infinite series of ``higher-order operations'' $l_i$ with $i\ge 1$. Those $\Linf$ algebras with $l_i=0$ for $i>2$ are precisely dgla's, in which case $l_1$ is the differential and $l_2$ is the Lie bracket. The Maurer-Cartan equations for $L_\infty$ algebras involve the higher operations $l_i$ as well. There is by now a well-established deformation theory in terms of $L_\infty$ algebras, due to Fukaya, Konstevich, Soibelman, Manetti, etc. Namely,  one can define a deformation functor via Maurer-Cartan elements modulo homotopy equivalence, two equivalent $\Linf$ algebras give the same deformation functor, and that this recovers the deformation theory via dgla's. Moreover, every dgla $C$ is $L_\infty$ equivalent to its cohomology $HC$ endowed with an $L_\infty$ structure canonical up to $\Linf$ isomorphism. Hence  $\cM_{\rho}$ is given by the $L_\infty$ Maurer-Cartan equations in the affine space $H^1C$ modulo homotopy equivalence. The quadraticity from the formal case is a reflection of the fact that in that case only the (quadratic) Lie bracket $l_2$ is involved in the Maurer-Cartan equations. The solutions of the $L_\infty$ Maurer-Cartan equations in $H^1C$, that is before taking the quotient modulo  homotopy equivalence, describe locally the {mini-versal deformation space}, or the {Kuranishi space}, of $\rho$.

By a {\it deformation problem with cohomology constraints} we mean that the objects to be deformed come with a cohomology theory and that we would like to describe  those deformations whose $i$ degree cohomology has dimension $\ge k$.  In the presence of a moduli space $\cM$ of objects with a cohomology theory, this type of problem is equivalent to describing inside $\cM$ the formal germs at $\rho\in\cM$ of the {\it cohomology jump loci}, also known as the {\it generalized Brill-Noether loci}, 
$$
\cV^i_k=\{\rho\in\cM\mid \dim H^i(\rho)\ge k\}
$$
endowed with the natural structure of subschemes of $\cM$, for all $i$ and $k$. 

Deligne's principle has been extended by Budur-Wang \cite{BW} to this type of problems: every deformation problem with cohomology constraints over a field of characteristic zero is controlled by a pair consisting of a dgla together with a dgl module, and two equivalent dgl pairs describe the same deformation problem with cohomology constraints.  

More precisely, given an object $\rho\in \cM$, one has an attached dgla $C$ controlling the germ $\cM_{\rho}$. $C$ is typically an endomorphism object, and as such, it comes with a natural dgl module $M$ on which it acts, and such that $H^iM=H^i(\rho)$. Then the pair $(C,M)$ describes locally at $\rho$  the cohomology jump loci $\cV^i_k$ for all $i$ and $k$ at once. This means that one can attach subfunctors $\Def^i_k(C,M)$ of $\Def(C)$ that are naturally isomorphic to the functors associated to the formal germs $(\cV^i_k)_\rho$ of $\cV^i_k$ at $\rho$. Moreover, for two equivalent dgl pairs, the cohomology jump subfunctors $\Def^i_k$ are isomorphic \cite{BW}.  

The approach via dgl pairs to cohomology jump loci is successful in the situation when the pair $(C,M)$ is equivalent to its cohomology pair $(HC,HM)$, that is, when  $(C,M)$ is a {\it formal} dgl pair. Typically $(HC,HM)$ is finite-dimensional. In addition, the vanishing of the differentials in $(HC,HM)$ translates into a beautiful answer for the deformation problem with cohomology constraints in terms of the linear algebra of the multiplication map $HC\otimes HM\rightarrow HM$: the cohomology jump loci $\cV^i_k$ are locally given at $\rho$ by  certain determinantal ideals, the so-called {\it cohomology jump ideals}, depending on $i$ and $k$ of the matrices of linear forms corresponding to the above multiplication map. In particular, if $\rho$ is generic in the sense that $\rho\in\cV^i_k\setminus\cV^i_{k+1}$, then locally at $\rho$ the locus $\cV^i_k$ is cut out by linear equations from the quadratic space $\cM_{\rho}$. Thus $(\cV^i_k)_\rho$ also has quadratic singularities in this case, see \cite{BW}.

The ``generic'' loci $\{\cV^i_k\setminus\cV^i_{k+1}\mid i\in \bZ\}$ --that is, deformations with the same cohomology-- are equivalently handled by Manetti's deformation functor $\Def_\chi$ {in  \cite{M-a}, where $\chi:C\ra \End (M)$ is a dgla morphism with $(C,M)$ a dgl pair.} However, the non-generic structure of the scheme $\cV^i_k$ is much more complicated even if generically it is not. In fact, in Theorem \ref{propTDef} we will see that away from the generic locus, the Zariski tangent spaces of $\cV^i_k$ are the full tangent space  of the ambient moduli.

\subsection{Cohomology jump loci via \texorpdfstring{$L_\infty$}- pairs.}
\label{subLin} 
How does one describe locally the cohomology jump loci over a field of characteristic zero in the absence of formality for the controlling dgl pair? We provide the  answer in terms of $L_\infty$ pairs. More precisely, we develop the theory of cohomology jump subfunctors $$\Def^i_k(L,M)\subset\Def(L)$$ for any pair $(L,M)$ consisting of an $L_\infty$ algebra $L$ together with an $L_\infty$ module $M$, such that it extends the theory for dgl pairs from \cite{BW}. This combines the higher $L_\infty$ multiplication maps with the theory of cohomology jump ideals. The main result is:

\begin{theorem}
\label{thmMainLDef} Let $(C,M)$ be a dgl pair over a field of characteristic zero. On the associated cohomology pair $(HC,HM)$ there exists an $\Linf$ pair structure with zero differentials and with second-order multiplication maps induced from $(C,M)$, unique up to isomorphisms of $\Linf$ pairs, such that $(C,M)$ and $(HC,HM)$ are $\Linf$ equivalent. Moreover, the cohomology jump subfunctors of the $\Linf$ pair $(HC,HM)$
$$
\Def^i_k(HC,HM)\subset \Def(HC)
$$
are naturally isomorphic to the cohomology jump subfunctors of the dgl pair $(C,M)$
$$
\Def^i_k(C,M)\subset \Def(C)
$$
for all $i,k\in\bZ$.
\end{theorem}

The upshot is: given a deformation problem with cohomology constraints 
\be\label{eqDPCC}
\{(\cV^i_k)_\rho\subset\cM_{\rho}\mid i,k\in\bZ \},
\ee 
let  $(C,M)$ be a dgl pair controlling it; then (\ref{eqDPCC}) is the same as the collections of subfunctors defined via the $\Linf$ pair structure
\be\label{eqDPSF}
\{\Def^i_k(HC,HM)\subset \Def(HC) \mid i,k\in\bZ\}.
\ee
If the cohomology pair $(HC,HM)$ is finite-dimensional, which is typically the case in applications, (\ref{eqDPSF}) provides  equations for (\ref{eqDPCC}) in terms of the cohomology jump ideals of the associated $L_\infty$ multiplication maps, up to homotopy equivalence.

We determine the Zariski tangent spaces to the cohomology jump functors:

\begin{theorem}
\label{propTDef}
Let $(C,M)$ be a dgl pair or, more generally, an $\Linf$ pair, over a field of characteristic zero. Let $h_i=\dim H^iM$. The Zariski tangent spaces to the functors 
$$\Def^i_0(C,M)=\Def(C)\supset\ldots\supset\Def^i_k(C,M)\supset\ldots \supset  \Def^i_{h_i+1}(C,M)=\emptyset$$ 
are: the full Zariski tangent space $T\Def(C)=H^1C$ if $k<h_i$; empty if $k>h_i$; and if $k=h_i$, equal to the kernel of the linear map
$$
H^1C\ra \bigoplus_{j=i-1,i}\Hom(H^{j}M,H^{j+1}M)
$$
induced from the module multiplication maps
$
H^1C\otimes H^jM\ra H^{j+1}M.
$
\end{theorem}

In particular, the Zariski tangent spaces of cohomology jump loci $\cV^i_k$ at points in $\cV^i_{k+1}$ are equal to the full Zariski tangent space of the ambient moduli space $\cM=\cV^i_0$. This statement is to our knowledge new for higher rank local systems, but also for vector bundles on higher dimensional varieties.

We use the $\Linf$ pairs theory to give a proof of an open problem on finite dimensional cdga's of Suciu \cite{Su}*{Problem 3.2}:

\begin{theorem}
\label{thrmGTC}
Let $(A=\bigoplus_{i\ge 0}A^i, d)$ be a cdga over the field $\bC$, with $A^0=\bC$ and all $A^i$ finite dimensional. Then the tangent cone at $0$ of the resonance variety $R^i_k(A)$ is contained in the resonance variety $R^{i}_k(H)$ of the cohomology cdga $H=H(A)$. 
\end{theorem}
Here $H^1$ is identified with the 1-cocycles of $(A,d)$, and 
\begin{align*}
R^i_k(A) & :=\{a\in H^1\mid \dim H^i(A,d+a)\ge k\},\\
R^i_k(H) & :=\{a\in H^1\mid \dim H^i(H,a)\ge k\}.
\end{align*}

In a sequel to this article, we will apply the $\Linf$ pairs theory developed here to Brill-Noether loci of vector bundles on smooth projective varieties.

%We also prove a filtered version of the $\Linf$ structure Transfer Theorem: \begin{theorem} Let $(C,M)$ be a filtered dgl pair. Then there always exists an $\Linf$ structure on $(HC,HM)$ compatible with the induced filtration, such that $(HC,HM)$ is $\Linf$ equivalent to $(C,M)$. \end{theorem}

% \subsection{Abbreviations} We use: {\it cdga} = differential graded commutative algebra; {\it dga} = differential graded  algebra; {\it dgla} = differential graded Lie algebra.
 
\subsection{Organization.} Section \ref{sec2} is contains the main definitions and properties of $\Linf$ pairs, the material being standard or straight-forward if new. In Section \ref{def functors section}, the technical core of the article, we introduce the cohomology jump functors of an $\Linf$ pair, and prove Theorems \ref{thmMainLDef}, \ref{propTDef}, and \ref{thrmGTC}. Section \ref{sec4} contains the applications to local systems, namely the proofs of Theorems \ref{corW}, \ref{thrmLink}, \ref{thrmMFiber}, and \ref{thrmLoc}.

\subsection{Acknowledgement.} We thank  D. Arapura, J. Cirici, R. Hain, J. Koll\'ar, Y. Liu, D. Petersen, M. Saito, B. Shoikhet, A. Suciu, S. Yalin, B. Wang, M. Zambon, and the referees for comments and discussions. The authors were partly supported by the grant STRT/13/005 from KU Leuven, the Methusalem grant METH/15/026, and the grants G0B2115N, G097819N, G0F4216N from FWO.

\section{\texorpdfstring{$L_\infty$}~ pairs}\label{sec2}

Throughout this paper $\bK$ is a fixed field of characteristic zero. $\texttt{Art}$ denotes the category of commutative Artinian local finite type $\bK$-algebras. All the vector spaces we consider are over $\bK$ unless mentioned otherwise. By a complex we mean a cochain complex.

\subsection{$\Linf$ algebras} We review basic notions, see \cites{LaMa94,Kon03, LodVall12}.

\begin{notation}
For a graded vector space $L=\bigoplus_{i\in\mathbb{N}}L^i$ and a homogeneous element $a\in L$, we denote by $|a|$ the degree of $a$. For two homogeneous elements $a,b\in L$, the Koszul sign of their transposition is defined as $(-1)^{|a||b|}$. More generally, for $n$ homogeneous elements $a_1,\ldots , a_n\in L$, and any $n$-permutation $\sigma$, the {\it Koszul sign} $\chi(\sigma)=\chi(\sigma; |a_1|,\ldots,|a_n|)$ is the product of the Koszul signs of the transpositions necessary to permute $(a_{\sigma(1)},\ldots, a_{\sigma(n)})$ to $(a_1,\ldots,a_n)$.
\end{notation}

\begin{definition}
A {\it graded antisymmetric multilinear map} on a graded vector space $L $ is a linear map of graded vector spaces $$l_n: L^{\otimes n} \rightarrow L$$  such that $$
l_n(a_{\sigma(1)}, \dots, a_{\sigma(n)})= \chi(\sigma)l_n( a_1, \dots, a_n)
$$
for every $n$-permutation $\sigma$ and homogeneous $a_i\in L$, where we write commas instead of tensor symbols between the $a_i$.
\end{definition}

\begin{definition}
An {\it $L_\infty$ algebra}  is a graded vector space $L$ together with a collection of  graded antisymmetric multilinear maps 
$$
\{l_n: L^{\otimes n} \rightarrow L\}_{n\geq1}
$$ called the higher-order $\Linf$ multiplication maps, such that $l_n$ has degree $2-n$ and the generalized Jacobi identity
$$
\sum_{i+j = n+1} \sum_{\sigma \in \Sigma{(i,n-i)}} \chi(\sigma)(-1)^{i(j-1)} l_j(l_i(a_{\sigma(1)}, \dots, a_{\sigma(i)}), a_{\sigma(i+1)}, \dots, a_{\sigma(n)}) = 0
$$
holds for any homogeneous elements $a_1,\ldots, a_n\in L$, where $\Sigma{(i, n-i)}$ is the set of $(i, n-i)$-unshuffles with $i\ge 1$, that is, the subset of $n$-permutations $\sigma$  such that $\sigma(1) < \cdots < \sigma(i)$ and $\sigma(i+1) <\cdots<\sigma(n)$. 
\end{definition}

For a simpler notation, we will denote 
$$\mathfrak{S}_{n}=\{   (i,j,\sigma) \mid  \sigma \in \Sigma(i,n-i), i\ge 1, \text{ and }  i+j = n+1\}$$ 
$$I = (1, \dots, n), \quad  a_{\sigma(I)} = (a_{\sigma(1)}, \dots, a_{\sigma(n)}),$$ and  the generalized Jacobi identity as:
$$
\sum_{(i,j,\sigma) \in \mathfrak{S}_{n}} \chi(\sigma)(-1)^{i(j-1)} l_j (l_i \otimes \textbf{1}^{\otimes (j-1)})(a_{\sigma(I)}) = 0.
$$

By definition, $(L,l_1)$ is a complex and we will denote by $HL$ the cohomology of $(L,l_1)$.

%%%remark on usage of trees?

%For the following definition, see for example \cite{LodVall12}*{Proposition 10.2.7}.

\begin{definition}\label{inf-morphism}
A {\it morphism of $L_\infty$ algebras} $f: L\ra L'$ is a collection of degree ${1-k}$   graded antisymmetric multilinear maps
$$
\{f_k: L^{\otimes k}\rightarrow L'\}_{ k \geq1}
$$
such that for any  $a_I=(a_1,\ldots,a_n)$ with $a_i\in L$ homogeneous and $n\ge 1$,
\begin{align*}
\sum_{(i,j,\sigma) \in \mathfrak{S}_{n}} \chi(\sigma) (-1)^{i(j-1)} & f_j(l_i \otimes \textbf{1}^{\otimes (j-1)})(a_{\sigma(I)}) = \\
& = \sum_{(k_1,\ldots,k_j,\tau) \in \mathfrak{S}_{j,n}}{\chi(\tau)}(-1)^\varepsilon l'_j(f_{k_1}\otimes \cdots \otimes f_{k_j})(a_{\tau(I)}),
\end{align*}
where: $l$ and $l'$ are the structure operations on $L$ and $L'$ respectively;
 $\mathfrak{S}_{j,n}$ is the set of tuples $(k_1,\ldots,k_j,\tau)$ with $k_i\ge 1$, $k_1+\ldots +k_j=n$, and $\tau$ is an $n$-permutation which preserves the order within each block of length $k_i$; and 
 $$
 \varepsilon = (j-1)(k_1-1)+(j-2)(k_2-1)+\ldots +2(k_{j-2}-1)+(k_{j-1}-1).
 $$
%In particular, in the right hand side, we have: \begin{displaymath} \begin{split}& \sum_{\tau \in \mathfrak{S}_{k_1+\cdots+k_j=n}}\frac{\chi(\tau)}{j!} l_j' (f_{k_1}(a_{\tau(1)},\dots,a_{\tau(k_1)}), f_{k_2}(a_{\tau(k_1+1)}, \dots a_{\tau(k_1+k_2)}),\dots \hspace{0.3cm}\\& \hspace{8cm}\dots, f_{k_j}(a_{\tau(k_n-k_{j-1}+1)}, \dots, a_{\tau(n)})).\end{split}\end{displaymath}
\end{definition}

\begin{definition}\label{defWE}
A morphism of $\Linf$ algebras $f: L \rightarrow L'$ is a {\it weak equivalence} if the map of complexes $f_1: (L,l_1) \rightarrow (L',l'_1)$ is  a quasi-isomorphism. 
\end{definition}

\subsection{$A_\infty$ algebras}\label{secAinf}

A source of $\Linf$ algebras are the $A_\infty$ algebras.  We follow \cites{Keller,LaMa94} to recall some definitions and facts. 

\begin{defn}
An \emph{$\Ainf$ algebra} is a $\mathbb{Z}$-graded vector space $A=\bigoplus_{r\in{\mathbb{N}}}A^r$ together with a collection of degree $2-n$ multilinear maps 
\[\left\{\nu_n: A^{\otimes n} \rightarrow A\right\}_{n\geq 1}\]
that satisfy the generalized associative relation:
\[\sum_{p+q+r=n}(-1)^{p+qr}\nu_{p+r+1}(\textbf{1}^{\otimes p}\otimes \nu_q\otimes \textbf{1}^{\otimes r})=0\]
for $n\geq 1$. When applied to elements, these formulas acquire extra signs according to the Koszul rule.

A \emph{morphism of $\Ainf$ algebras} $f:(A,\nu)\rightarrow (B,\nu')$ is a collection of degree $1-n$ multilinear maps
\[\left\{f_n: A^{\otimes n} \longrightarrow B\right\}_{n\geq 1}\]
satisfying the relation
\[\sum_{\substack{p+q+r=n\\p+r+1=k}} (-1)^{p+qr} f_{k} (\textbf{1}^{\otimes p} \otimes \nu_{q}\otimes \textbf{1}^{\otimes r}) = \sum_{i_1+\cdots+i_k=n} (-1)^{\varepsilon}\nu'_{k}(f_{i_1}\otimes\dots\otimes f_{i_k})\]
for $k,n \geq 1$, where $\varepsilon$ is the same as in Definition \ref{inf-morphism}, i.e.
\[ \varepsilon = (k-1)(i_1-1)+(k-2)(i_2-1)+\cdots+2(i_{k-2}-1)+(i_{k-1}-1).\]
We say that an $\Ainf$ morphism is a \emph{weak equivalence} if $f_1$ is a quasi-isomorphism.
\end{defn}

%\begin{defn} Let $(A,\nu)$ be an $\Ainf$ algebra. A (left) $\Ainf$ module over $A$ is a differential graded vector space $(M,m_1)$ together with a collection of graded linear maps \[\left\{m_n: A^{\otimes (n-1)}\otimes M \longrightarrow M\right\}_{n\geq 1}\] of degree $2-n$ such that the following relation \[\sum_{p+q+r=n}(-1)^{r-pq} m_{p+r+1}(\textbf{1}^{\otimes p}\otimes \nu_q\otimes \textbf{1}^{\otimes r})=0\] holds for $r>0$ and \[\sum_{p+q=n}(-1)^{pq} m_{p+1}(\textbf{1}^{\otimes p}\otimes m_q)=0\] holds for $r=0$. A \emph{morphism of $\Ainf$ modules over $A$} $f: (M, m) \rightarrow (N, m')$ is a collection of degree $1-n$ multilinear maps \[\left\{f_n: A^{\otimes (n-1)}\otimes M \longrightarrow N\right\}_{n\geq 1}\] such that for each $n\geq 1$ we have \[\sum_{\substack{r+s+t=n\\r+t+1=u}} (-1)^{rs+t}f_u (\textbf{1}^{\otimes r}\otimes m_s \otimes \textbf{1}^{\otimes t}) = \sum_{s+t=n}(-1)^{s(t+1)} m'_{s+1}(\textbf{1}^{\otimes s}\otimes f_t),\] with $r,t \geq 0$ and $s\geq 0$. Again, we call such a morphism a \emph{weak equivalence} if $f_1$ is a quasi-isomorphism.\end{defn}

\begin{prop}
\label{Ainf to Linf}\cite{LaMa94}
Given an $\Ainf$ algebra structure $\{\nu_n\}$ on a graded vector space $A$, one can associate to it an $\Linf$ algebra structure $\{l_n\}$ on $A$: for $a_1,\dots,a_n\in A$,
\[ l_n(a_1,\dots,a_n) :=  \sum_{\sigma}\chi(\sigma)\nu_n(a_{\sigma(1)},\dots,a_{\sigma(n)}),\]
where the sum is over all $n$-permutations. This correspondence defines a functor from the category of $\Ainf$ algebras to that of $\Linf$ algebras preserving weak equivalences.
\end{prop}

\subsection{$\Linf$ modules}

\begin{definition}
\label{L-module} Let $L$ be an $\Linf$ algebra with $l=\{l_n\}_{n\ge 1}$ as structure maps. Let $M$ be a differential graded vector space with differential denoted $m_1$; in this article we always assume that $M$ is bounded above, and has finitely generated cohomology. A structure of (left) {\it $\Linf$ module} over $L$ on $M$ is a collection of  graded linear maps 
$$
\{m_n: L^{\otimes (n-1)}\otimes M \rightarrow M\}_{n\geq1}
$$
such that $m_n$ has degree $2-n$ and 
$$
\sum_{(i,j,\sigma) \in \mathfrak{S}_{n}} \chi(\sigma)(-1)^{i(j-1)}m_j(m_i\otimes \textbf{1}^{\otimes (j-1)})(\xi_{\sigma(I)})=0,
$$
for every homogeneous elements $\xi_1, \dots, \xi_{n-1} \in L$ and  $\xi_n \in M$, subject to the following convention. Note first that by definition if $(i,j,\sigma)\in \mathfrak{S}_{n}$, then either $\sigma(i)=n$ or $\sigma(n)=n$. In the first case one defines
$$
m_j(m_i\otimes \textbf{1}^{\otimes (j-1)})(\xi_{\sigma(I)})=\kappa\cdot m_j(\textbf{1}^{\otimes (j-1)}\otimes m_i)(\xi_{\sigma(i+1)},\ldots, \xi_{\sigma (n)},\xi_{\sigma (1)},\ldots ,\xi_{\sigma (i)}),$$ where $$
\kappa = (-1)^{j-1} \cdot (-1)^{(i+|\xi_{\sigma(1)}|+\cdots+|\xi_{\sigma(i)}|)\cdot(|\xi_{\sigma(i+1)}|+\cdots+|\xi_{\sigma(n)}|)}
$$ according to the Koszul sign convention. In the second case we take $m_i=l_i$.

By definition, $(M,m_1)$ is a complex. We will denote by $HM$ the cohomology of $(M,m_1)$ as a graded vector space.
\end{definition}

\begin{definition}
\label{defMLM}  
A {\it morphism of $\Linf$ modules} $f: (M, m) \rightarrow (M', m')$ between two $L$-modules  is a collection of degree ${1-k}$ graded antisymmetric multilinear maps
$$
\{f_k: L^{\otimes (k-1)}\otimes M \rightarrow M'\}_{k \geq1}
$$
such that  for $n\geq1$
\begin{align*}
\sum_{(i,j,\sigma) \in \mathfrak{S}_{n}} \chi(\sigma)(-1)^{i(j-1)} f_j & (m_i\otimes \textbf{1}^{\otimes (j-1)})(\xi_{\sigma(I)}) =\\
& \sum_{(k_1,\ldots,k_j,\tau) \in \mathfrak{S}_{j,n}}{\chi(\tau)}(-1)^\varepsilon m'_j(f_{k_1}\otimes \cdots \otimes f_{k_j})(\xi_{\tau(I)})
\end{align*}
for all homogeneous $\xi_1,\ldots,\xi_{n-1}\in L$ and $\xi_n\in M$, with the conventions as in Definition \ref{L-module}, and with $\varepsilon$ as in Definition \ref{inf-morphism}.
\end{definition}

%Finally, we say that an $L$-module $M$ is \textbf{formal} if it is weakly equivalent to its cohomology seen as an $H^\ubul(L)$-module $(H^\ubul(M), 0, r_2, 0, 0, \dots)$, where $r_2$ is the restriction of $m_2$ in $H^\ubul(M)$.

%The next theorem is an $\Linf$ analogue of a classical result on Lie algebras that given a Lie algebra $L$ and an $L$-module $M$, {the vector space $L\oplus M$ forms a Lie algebra via the bracket \[\left[(x_1, m_1),(x_2, m_2)\right]_{L\oplus M}= \big( [x_1, x_2]_L, x_1 \cdot m_2 - x_2\cdot m_1 \big).\]}

\begin{theorem} \cite{Lad04} \label{L oplus M structure}
Let $L$ be an $L_\infty$ algebra and $M$ an $L$-module. Then, the graded vector space $L\oplus M$ inherits a canonical $L_\infty$ structure given by the collection of degree $2-n$ graded antisymmetric multilinear maps
$$
\{ j_n : (L\oplus M)^{\otimes n} \rightarrow L\oplus M\}_{n\geq1},
$$
satisfying the following relation for homogeneous elements $(a_i,\xi_i)\in L\oplus M$:
\begin{displaymath}
\begin{split}
&j_n\big((a_1, \xi_1), \dots, (a_n, \xi_n)\big) \hspace{5cm}\\
&=\left(l_n(a_1, \dots, a_n), \sum_{i=1}^{n} (-1)^{n-i+|\xi_i| \sum_{k=i+1}^{n} |a_k|} m_n(a_1, \dots, \hat{a_i}, \dots, a_n, \xi_i)\right),
\end{split}
\end{displaymath}
where $\hat{a_i}$ refers to omitting $a_i$ in the list.
\end{theorem}

It is sometimes convenient to switch from this algebra back to the module:

\begin{proposition}\label{module structure recovered} Let $L$ be an $L_\infty$ algebra and $M$ an $L$-module. The structure maps $j_n$ of the $L_\infty$-algebra $L\oplus M$ as above recover the structure maps $m_n$ of $M$.
\end{proposition}

\begin{proof}
We take the map $j_n: (L\oplus M)^{\otimes n} \rightarrow L\oplus M$ and consider the assignment $$j_n\big((a_1, \xi), (a_2, \xi), \dots, (a_{n-1},\xi), (0,\xi)\big)\in L\oplus M.$$
Then, by Theorem \ref{L oplus M structure}, 
\begin{displaymath}
\begin{split}
& j_n\big((a_1, \xi), (a_2, \xi), \dots, (a_{n-1},\xi), (0,\xi)\big) \hspace{3.8cm}\\
&=\left(l_n(a_1, \dots, a_{n-1},0), \sum_{i=1}^{n} (-1)^{n-i}(-1)^{|\xi| \sum_{k=i+1}^{n} |a_k|} m_n(a_1, \dots, \hat{a_i}, \dots, a_{n-1}, 0, \xi)\right).
\end{split}
\end{displaymath}
Hence the first component vanishes, and the second component collects the only possibly-nonzero summand when $i=n$. Thus one obtains the following element in $L\oplus M$: 
$$
\big(0, m_n(a_1, \dots, a_{n-1}, \xi)\big).
$$
If we restrict it to the second component, we recover the structure map $m_n$.
\end{proof}

\begin{definition} A  morphism  of $L$-modules $f: M \rightarrow N$ is a {\it weak equivalence} if its first component $f_1: (M, m_1) \rightarrow (N, n_1)$ is a quasi-isomorphism of complexes.\end{definition}

\subsection{{$L_\infty$} pairs}

\begin{definition}\label{defLpair}
An {\it $L_\infty$ pair} is an $L_\infty$ algebra $L$ together with an $L$-module $M$. We denote such pair by $(L,M)$. A {\it morphism of $L_\infty$ pairs} between $(L,M)$ and $(L',M')$ is a tuple $(f,g)$ where $f:L\rightarrow L'$ is a morphism of $\Linf$ algebras and $g: M \rightarrow M'$ is a morphism of $L$-modules, where $M'$ is regarded as an $L$-module via $f$.  We say that a morphism $(f,g)$ of $\Linf$ pairs is {a} {\it weak equivalence} if $f$ and and $g$ are  weak equivalences. 
%An $L_\infty$ pair $(L, M)$ is {\it formal} if $(L,M)$ is  weakly equivalent to the cohomology pair $\big((HL,0, l_2, 0, \dots), (HM,0, m_2, 0, \dots) \big)$ with $l_2$ and $m_2$ inherited from $(L,M)$. 
\end{definition}

\begin{prop}\label{formal pair iff formal oplus}
(i) A morphism of $\Linf$ pairs $(f,g):(L,M)\rightarrow (L',M')$ induces a morphism of $\Linf$ algebras $f\oplus g: L\oplus M \rightarrow L'\oplus M'$.

(ii) The morphism of $\Linf$ algebras $f\oplus g$ recovers the  morphism of $\Linf$ pairs $(f,g)$.

(iii) $(f,g)$ is a weak equivalence if and only if $f\oplus g$ is a weak equivalence.
%(iv) An $L_\infty$ pair $(L,M)$ is formal if and only if $L\oplus M$ is a formal $L_\infty$ algebra.
\end{prop}

\begin{proof}

For (i) let $(f,g):(L,M)\rightarrow (L',M')$ be a morphism of $L_\infty$ pairs. We show that there exists a morphism of $L_\infty$ algebras $f\oplus g: L\oplus M\rightarrow L'\oplus M'$, that is, a collection of degree $1-k$ graded antisymmetric multilinear maps
\[(f\oplus g)_k : (L\oplus M)^{\otimes k} \longrightarrow (L',M')\]
such that they satisfy Definition \ref{inf-morphism}. Define for $n$ homogeneous elements $(a_i,\xi_i)\in L\oplus M$
\begin{align}\label{oplus1}
(f\oplus g)_n (a\oplus \xi) := \left(f_n(a_1,\dots, a_n), \sum_{i=1}^n (-1)^{\Xi(n,i)} g_n (a_1,\dots, \hat{a_i},\dots, a_n, \xi_i)\right),
\end{align}
where ${\Xi(n,i)} := {n-i+|\xi_i| \sum_{k=i+1}^{n} |a_k|}$ gives the appropriate sign when omitting one element and shifting the remaining to the left. To satisfy Definition \ref{inf-morphism} we need that 
\begin{align}\label{oplus2}
\sum_{(i,s,\sigma) \in \mathfrak{S}_{n}} \chi(\sigma) & (-1)^{i(s-1)}  (f\oplus g)_s(j_i \otimes \textbf{1}_{L\oplus M}^{\otimes (s-1)})((a,\xi)_{\sigma(I)}) \notag \\
& = \sum_{(k_1,\ldots,k_s,\tau) \in \mathfrak{S}_{s,n}}{\chi(\tau)}(-1)^\varepsilon j'_s(((f\oplus g)_{k_1}) \otimes \cdots \otimes ((f\oplus g)_{k_s}))((a,\xi)_{\tau(I)}),
\end{align}
where $j$ and $j'$ are the $L_\infty$ algebra structures on $L\oplus M$ and $L'\oplus M'$, respectively, as constructed in Theorem \ref{L oplus M structure}.

Denote by $(l,m)$ the $\Linf$ pair structure on $(L,M)$, and by $(l',m')$ that on $(L',M')$. Looking closer at the left hand side of (\ref{oplus2}), we have
\begin{align*}
\sum_{(i,s,\sigma) \in \mathfrak{S}_{n}} \chi(\sigma) (-1)^{i(s-1)} (f\oplus g)_s  \bigg(\bigg( & l_i(a_{\sigma(1)},\dots, a_{\sigma(i)}),\\
& \sum_{q=1}^{i}(-1)^{\Xi(i,q)} m_i (a_{\sigma(1)},\dots,\hat{a}_{\sigma(q)},\dots, a_{\sigma(i)},\xi_{\sigma(q)})\bigg),\\
& (a_{\sigma(i+1)},\xi_{\sigma(i+1)}),\dots,(a_{\sigma(n)},\xi_{\sigma(n)})\bigg)=
\end{align*}
\begin{align*}
=\sum_{(i,s,\sigma) \in \mathfrak{S}_{n}} & \chi(\sigma) (-1)^{i(s-1)}  \bigg( f_s\bigg(l_i(a_{\sigma(1)},\dots,a_{\sigma(i)}), a_{\sigma(i+1)},\dots,a_{\sigma(n)}\bigg), \\
&\sum_{p=1}^s (-1)^{\Xi(s,p)} g_s \bigg(l_i(a_{\sigma(1)}, \dots,a_{\sigma(i)}),a_{\sigma(i+1)},\dots, \hat a_{\sigma(p)} ,\dots, a_{\sigma(n)}, \xi_{\sigma(p)} \bigg)\bigg).
\end{align*}
Using the fact that $f$ is an $\Linf$ morphism and $g$ an $L$-module morphism, this is equal to
\begin{align*}
\sum_{(k_1,\ldots,k_s,\tau) \in \mathfrak{S}_{s,n}} & {\chi(\tau)}(-1)^\varepsilon 
\bigg({l'}_s(f_{k_1}\otimes \cdots \otimes f_{k_s})(a_{\tau(I)}),\\ 
& \sum_{p=1}^s (-1)^{\Xi(s,p)} {m'}_s (g_{k_1} \otimes \cdots \otimes g_{k_s})(a_{\tau(1)},\dots,\hat a_{\tau(p)}, \dots, a_{\tau(n)}, \xi_{\tau(p)}) \bigg).
\end{align*}
By Theorem \ref{L oplus M structure}, this is the $\Linf$ algebra structure on $L'\oplus M'$, that is, we get the right hand side of (\ref{oplus2}). Thus $f\oplus g$ is a morphism of $\Linf$ algebras.

To prove (ii), note that restricting $f\oplus g$ to the first component in (\ref{oplus1}), gives $f$ back. To recover $g$ we use the same method as in Proposition \ref{module structure recovered}. Namely, consider the assignment
\[(f\oplus g)_n\big((a_1, \xi), (a_2, \xi), \dots, (a_{n-1},\xi), (0,\xi)\big)\in L'\oplus M'.\]
Then, by (\ref{oplus1}), this is  equal to
\[=\left(f_n(a_1,\dots, 0), \sum_{i=1}^n (-1)^{\Xi(n,i)} g_n (a_1,\dots, \hat{a_i},\dots, 0, \xi)\right).\]
Hence, the first component vanishes, and the second component collects the only possible nonzero summands when $i=n$. Thus one obtains the following element in $L'\oplus M'$:
\[ \left( 0, g_n(a_1, \dots, a_{n-1},\xi)\right).\]
Restricting to the second component, we recover the map $g_n$.

Part (iii) follows from Definition \ref{defLpair}. If $(f,g)$ is a weak equivalence, we have that both $f_1$ and $g_1$ are quasi-isomorphisms. By (\ref{oplus1}), $(f\oplus g)_1$ is also a quasi-isomorphism, thus making $f\oplus g$ a weak equivalence by Definition \ref{defWE}. Conversely, let $f\oplus g$ be a weak equivalence. Then again this means that $(f\oplus g)_1$ is a quasi-isomorphism. Since we can restrict at either component of (\ref{oplus1}), we have that both $f_1$ and $g_1$ are quasi-isomorphisms; thus making $(f,g)$ a weak equivalence.
%Finally we prove (iv). By part (iii) a morphism $(f,g)$ of $\Linf$ pairs, with $f:L\rightarrow (H(L), 0, l_2, 0, \dots)$ and $g:M \rightarrow (H(M), 0, m_2, 0, \dots)$, is a weak equivalence if and only if $f\oplus g: L\oplus M \rightarrow (H(L)\oplus H(M), 0, j_1, 0, \dots)$ is a weak equivalence. By Definitions \ref{defLpair} and \ref{defWE}, this means that $(L,M)$ is formal if and only if $L\oplus M$ is formal.
\end{proof}

\begin{rmk} The category of dgl pairs of \cite{BW} is a subcategory of the category of $\Linf$ pairs, but not a full subcategory.
\end{rmk}

\subsection{Transfer theorem for $A_\infty$ algebras with extra gradings.}\label{subAtr} For the long history behind the transfer theorem see \cite{Hu}.

\begin{defn}\label{defnTD}
A {\it homotopy transfer diagram} over $\bK$ is a diagram 
$$
\xymatrix{
A \ar@(lu,ld)_h \ar@<.5ex>[r]^f & \ar@<.5ex>[l]^g B
}
$$
together with the following data:
\begin{itemize}
\item $(A,\mu_1)$ and $(B,\nu_1)$ are cochain complexes of $\bK$-vector spaces,
\item $f:(A,\mu_1)\ra(B,\nu_1)$ and $g:(B,\nu_1)\ra (A,\mu_1)$ are morphisms of complexes,
\item $h$ is a collection of linear maps $h^n:A^n\ra A^{n-1}$ such that $\bone_A-gf=\mu_1h+h\mu_1$.
\end{itemize}
\end{defn}

\begin{defn}
Given a  homotopy transfer diagram as above with $(A,\mu_1,\mu_2,\mu_3,\ldots)$  an $A_\infty$ algebra, the associated {\it $p$-kernels}
$$
p_n:A^{\otimes n}\ra A \quad{(n\ge 2)}
$$
are the linear maps of degree $2-n$ defined inductively as
$$
p_n:=\sum_{B(n)} (-1)^{\theta(r_1,\ldots,r_k)} \mu_k((h\circ p_{r_1}) \otimes \ldots \otimes (h\circ p_{r_k})),
$$
where 
$$h\circ p_1 := \bone_A,$$
$$
\theta(r_1,\ldots,r_k):=\sum_{1\le i< j\le k}r_i(r_j+1), \text{ and }
$$
$$
B(n):=\{(k,r_1,\ldots,r_k)\mid k\ge 2, r_1, \ldots, r_k\ge 1, r_1+\ldots +r_k=n\}.
$$
The associated {\it $q$-kernels} 
$$
q_n:A^{\otimes n}\ra A \quad (n\ge 1)$$
are the linear maps of degree $1-n$ defined inductively by $q_1:=\bone_A$, and for $n\ge 2$ by
$$
q_n:=\sum_{C(n)}(-1)^{n+r_i+\theta(r_1,\ldots, r_i)}\mu_k((\psi\phi)_{r_1}\otimes\ldots\otimes(\psi\phi)_{r_{i-1}}\otimes(h\circ q_{r_i})\otimes\bone_A^{k-i})
$$
where
$$
C(n):=\{(k,i,r_1,\ldots,r_i)\mid k,i,r_1,\ldots,r_i\in\bN, 2\le k\le n, 1\le i\le k,$$
$$ r_1,\ldots, r_i\ge 1, r_1+\ldots r_i+k-i=n\},
$$
$$
(\psi\phi)_m:=gf\circ q_m+\sum_{B(m)}(-1)^{\theta(r_1,\ldots,r_k)} (h\circ p_k)((gf\circ q_{r_1})\otimes\ldots\otimes(gf\circ q_{r_k})).
$$
\end{defn}

One has the following explicit form of the $A_\infty$ homotopy transfer theorem, after \cite{Mark} and \cite{Kopriva}*{\S 3}:

\begin{theorem}\label{thrmEHTT} Let $$
\xymatrix{
A \ar@(lu,ld)_h \ar@<.5ex>[r]^f & \ar@<.5ex>[l]^g B
}
$$
be a homotopy transfer diagram with $(A,\mu_1,\mu_2,\mu_3,\ldots)$  an $A_\infty$ algebra, and consider the associated $p$-kernels and $q$-kernels. Define
$$
\nu_n:=f\circ p_n\circ g^{\otimes n}, \quad \phi_n:=f\circ q_n,\quad \psi_n:=h\circ p_n\circ g^{\otimes n}, \quad H_n:=h\circ q_n.
$$
Then 
\begin{itemize}
\item $(B,\nu_1,\nu_2,\nu_3,\ldots)$ is an $A_\infty$ algebra structure on $(B,\nu_1)$,
\item $\psi=(g,\psi_2,\psi_3,\ldots)$ is an $A_\infty$ morphism from $(B,\mu)$ to $(A,\nu)$,
\item $\phi=(f,\phi_2,\phi_3,\ldots)$ is an $A_\infty$ morphism from $(A,\nu)$ to $(B,\nu)$,
\item the composition $\psi\phi$ is $A_\infty$ homotopy equivalent to $\bone_A$ via $H=(h,H_2,H_3,\ldots)$.
\end{itemize}
If in addition, $fg$ is also homotopy equivalent to $\bone_B$ (so, $f$ and $g$ are homotopy equivalences of each other), then $\psi$ and $\phi$ are weak equivalences.
\end{theorem}

We will not recall here what a homotopy is between two $A_\infty$ morphisms; see {\it loc. cit.} for definition. We will only use homotopy transfer diagrams where $f$ and $g$ are homotopy equivalences, so that the last conclusion holds. 

%The theorem implies the minimality theorem of \cite{Kad80}:\begin{theorem} \label{thrmK} (Kadeishvili) Let $A$ be a dga. Let $H$ be the cohomology $A$. Then there exists an $A_\infty$ algebra structure $(H,0,\nu_2, \nu_3,\ldots)$ with $\nu_2$ given by the multiplication of $A$, and  an $A_\infty$ morphism $H\xra{\sim} A$ which is a quasi-isomorphism. \end{theorem}
%{\begin{proof} We recall the proof. Let $d$ be the differential on $A$.  Let $Z=\ker d$, $B=d(A)$, {$K=A/Z.$}  Let $Z^n=H^n\oplus B^n$ be a splitting of $0\ra B^n\ra Z^n\ra H^n\ra 0$ and $A^n=Z^n\oplus K^n$ a splitting of $0\ra Z^n\ra A^n\ra K^n\ra 0.$ The total splitting $A^n=H^n\oplus B^n\oplus K^n$ defines a projection $f:A\ra H$ and an injection $g:H\ra A$ which are morphisms of complexes, together with a homotopy $h:A\ra A$ such that $\bone_A-gf=dh+hd$. Here $h(\al,d(\beta),\gamma)=(0,0,\beta)$ for $\al\in H^n$, $\beta\in K^{n-1}, \gamma\in K^n$. Note that $d$ gives a bijection between $K^{n-1}$ and $B^n$. The trio $(f,g,h)$ is thus a homotopy transfer diagram between the complexes $(A,d)$ and $(H,0)$. In addition, $f$ and $g$ are homotopy equivalences since $fg=\bone_H$. So Theorem \ref{thrmEHTT} applies. \end{proof}}

\begin{corollary}
Let $$
\xymatrix{
A \ar@(lu,ld)_h \ar@<.5ex>[r]^f & \ar@<.5ex>[l]^g B
}
$$
be a homotopy transfer diagram, where $(A,\mu_1,\mu_2,\mu_3,\ldots)$ is an $A_\infty$ algebra and $(B,\nu_1)$ is a complex. Suppose in addition that all the  components $A^i$ and $B^i$ admit an extra grading
$$
A^i=\bigoplus_pA^i_p\quad\text{and}\quad B^i=\bigoplus_pB^i_p
$$
such that all $\mu_n$, $\nu_1$, $f$, $g$, and $h$ are compatible with the extra grading, that is, the restrictions of these maps to a multigraded component have image in the ``right'' bigraded component as described by:
$$
\mu_n: A^{i_1}_{p_1}\otimes \ldots\otimes A^{i_n}_{p_n}\ra A^{i_1+\ldots+i_n+2-n}_{p_1+\ldots+p_n}
\quad \text{and} \quad
\nu_1:B^i_p\ra B^{i+1}_p,
$$
$$
f: A^i_p\ra B^i_p,\quad g:B^i_p\ra A^i_p,\quad h:A^i_p\ra A^{i-1}_p.
$$
Then all the maps $\nu_n$, $\phi_n$, $\psi_n$, $H_n$ from Theorem \ref{thrmEHTT} are also compatible with the extra grading, that is, their restrictions to a multigraded component have image in the ``right'' multigraded component, e.g.
\be\label{eqMulti}
\nu_n: B^{i_1}_{p_1}\otimes \ldots\otimes B^{i_n}_{p_n}\ra B^{i_1+\ldots+i_n+2-n}_{p_1+\ldots+p_n}.
\ee
\end{corollary}
\begin{proof}
Since the compositions and tensor products of graded maps are graded, it suffices  to prove that the $p$-kernels $p_n$ and the $q$-kernels $q_n$ are compatible with the extra grading. The same observation applies now to the inductive definition of $p_n$ and $q_n$.
\end{proof}

The following is an extension of Kadeishvili's minimality theorem \cite{Kad80} to extra gradings:

\begin{corollary}\label{corKgr}
Let $A$ be a dga with an extra grading on each $A^i$, compatible with the differential and the multiplication. Let $H$ be the cohomology $A$, with the induced extra grading on each $H^i$. Then there exists an $A_\infty$ algebra structure $(H,0,\nu_2, \nu_3,\ldots)$  and  an $A_\infty$ quasi-isomorphism $\psi:H\xra{\sim} A$, such that all $\nu_n$ and  $\psi$ are compatible with the multigrading induced by the extra grading, as in (\ref{eqMulti}).
\end{corollary}
\begin{proof} The proof is a just a graded version of the proof of Kadeishvili's theorem.
Let $$Z_p=\ker d_p,\quad B_p=d_p(A_p),\quad {K_p=A_p/Z_p},$$
where $d_p^n:A_p^n\ra A^{n+1}_p$ are the components with respect to the extra grading of the differential $d$ of $A$. Let $Z_p^n=H_p^n\oplus B_p^n$ be a splitting of $$0\ra B_p^n\ra Z_p^n\ra H_p^n\ra 0$$ and $A_p^n=Z_p^n\oplus K_p^n$ a splitting of $$0\ra Z_p^n\ra A_p^n\ra K_p^n\ra 0.$$ The  splittings $A_p^n=H_p^n\oplus B_p^n\oplus K_p^n$ define a projection $f_p:A_p\ra H_p$ and an injection $g_p:H_p\ra A_p$ which are morphisms of complexes, together with a homotopy $h_p:A_p\ra A_p[-1]$ such that $\bone_A-gf=dh+hd$ with $f=\bigoplus_pf_p$, $g=\bigoplus_pg_p$, $h=\bigoplus_ph_p$. The triple $(f,g,h)$ is thus a homotopy transfer diagram between the complexes $(A,d)$ and $(H,0)$, such that the maps $f$, $g$, $h$, $d$, and the multiplication on $A$, are compatible with the extra grading. So the previous corollary applies.
\end{proof}

\subsection{Transfer theorem for $\Linf$ pairs}\label{subTT} Similarly to Theorem \ref{thrmEHTT}, one has the following,  after \cite{LodVall12}*{Theorem 10.3.5}. We will not need the precise signs.

\begin{theorem}\label{minimal model of Linf}
Let $\left\{l_n:A^{\otimes n}\rightarrow A\right\}_{n\geq 1}$ be an $\Linf$ algebra structure on a graded vector space $A$. Given a {homotopy transfer diagram}
\[
\xymatrix{
A \ar@(lu,ld)_h \ar@<.5ex>[r]^f & \ar@<.5ex>[l]^g B,
}
\]
between complexes $(A,l_1)$ and $(B,l_1')$,
suppose in addition that $g$ is a quasi-isomorphism. Then there exists an  $\Linf$ algebra structure $\left\{l'_n:B^{\otimes n}\rightarrow B\right\}_{n\geq 1}$ on $(B,l'_1)$,
$l'_n=f\circ p_n\circ g^{\otimes n}$,
with $p_n$ defined inductively  from $l$ and $h$,
and there is  a weak equivalence of $\Linf$ algebras  $(A,l)\ra(B,l')$.
\end{theorem}

\begin{rmk} Alternatively, one can describe $p_n$ as a sum over rooted trees $\phi$ of $n$ leaves
\[p_n = \sum_{\phi }\pm  \phi(l,h).\]
The notation $\phi(l, h)$ stands for the multilinear operation on $A^{\otimes n}$ defined by the tree $\phi$ together with $l$ and $h$; for the definition we refer to \cite{LodVall12} for example.
\end{rmk}

\begin{cor}\label{minimal model of dgla}
Let $(C,d,\left[\_\;,\_\right])$ be a dgla and $(H,0)$ its cohomology. Given a  homotopy transfer diagram
\[
\xymatrix{
C \ar@(lu,ld)_h \ar@<.5ex>[r]^f & \ar@<.5ex>[l]^g H
}
\]
with $g$ a quasi-isomorphism,
there is an  $\Linf$ algebra structure  
$\left\{\mu_n:H^{\otimes n}\rightarrow H\right\}_{n\geq 2}$ on $(H,0)$,  $\mu_n=f\circ p_n\circ g^{\otimes n}$, with $p_n$ defined inductively from $d$ and $[\_,\_]$, and  there is a weak equivalence of $\Linf$ algebras
\[(H,0,\mu_2,\mu_3, \mu_4, \dots) \xlongrightarrow{\sim} (C,d,\left[\_\;,\_\right]).\]
\end{cor}

\begin{rmk} Alternatively, $p_n$ can  be described as a sum over rooted trees spanned by binary trees, with $n$ leaves, see \cite{LodVall12}:
\[p_n= \sum_{\phi} \pm \phi(\left[\_\;,\_\right],h).\]
\end{rmk}

\begin{theorem}\label{thmTTP}
Let $(C,M)$ be a dgl pair. Then there exists an $\Linf$ pair structure on the cohomology pair $(HC,HM)$ with zero differentials, and the second order operations inherited from $(C,M)$, together with a weak equivalence of $\Linf$ pairs $(HC,HM)\xrightarrow{\sim}(C,M)$. 
\end{theorem}
\begin{proof}
Given a dgl pair $(C,M)$, we consider the dgla $C\oplus M$ with differential $d_{C\oplus M}$ and bracket $[\_\;,\_]_{C\oplus M}$ given by the construction from Theorem \ref{L oplus M structure}. Then, by Corollary \ref{minimal model of dgla} there is $\Linf$ algebra structure $\mu^{H(C\oplus M)}$ on $H(C\oplus M)$ and a weak equivalence
\[(H(C\oplus M),\mu^{H(C\oplus M)})\xrightarrow{\sim}(C\oplus M, d_{C\oplus M}, [\_\;,\_]_{C\oplus M}).\]
Proposition \ref{formal pair iff formal oplus} (iii) gives us the desired weak equivalence of $\Linf$ pairs
$(HC, HM) \xrightarrow{\sim} (C,M).$
\end{proof}

\section{Cohomology jump functors of \texorpdfstring{$L_\infty$}~ pairs}\label{def functors section}

We recall the deformation functors of $\Linf$ algebras. We then refine them by cohomology jump subfunctors of $\Linf$ pairs, extending the case of dgl pairs from \cite{BW}. Then we prove Theorems \ref{thmMainLDef}, \ref{propTDef}, and \ref{thrmGTC}.

\subsection{Deformation functors of $\Linf$ algebras}

We  recall some standard facts using as reference \cites{Man04, Ge, LodVall12}.

%\begin{definition}\label{L infty MC elt} Let $L$ be an $\Linf$ algebra with structure operations $l=\{l_n\}_{n\ge 1}$. A {\it Maurer-Cartan element} in $L$ is an element of degree $1$, $\omega \in L^1$, that satisfies the Maurer-Cartan equation $$\sum_{n\geq 1} \frac{1}{n!} l_n(\omega^{\otimes n}) =0.$$We call $\MC_L$ the set of all Maurer-Cartan elements in $L$.\end{definition}

%Since we want to construct a functor from $\texttt{Art}_{\mathbb{K}}$ to \texttt{Set}, we consider $\MC_L(-):= \MC(L\otimes -)$. This is well defined for all $A\in\texttt{Art}_{\mathbb{K}}$, as $L\otimes A$ will be a new $L_\infty$ algebra with the underlying graded vector space of $L$ tensored by $A$. The new structure has the form $l_n^A = l_n \otimes \id_A$.

\begin{definition}\label{L infty MC functor}
Let $L$ be an $L_\infty$-algebra with structure maps $l=\{l_n\}$. The associated {\it Maurer-Cartan functor} of $L$ is the covariant functor
$$
\MC_L : \texttt{Art} \rightarrow \texttt{Set}
$$
defined for all Artinian local rings $(A, \mathfrak{m}_A) \in \texttt{Art}$ by
$$
\MC_L(A) = \left\{ \omega \in L^1\otimes\mathfrak{m}_A \hspace{0.2cm} \middle|\hspace{0.2cm} \sum_{n\geq1}\frac{1}{n!} \; l_n^A(\omega^{\otimes n})=0 \right\}
$$
where 
$$
l_n^A(a_1\otimes x_1, \dots, a_n\otimes x_n) = l_n(a_1, \dots, a_n)\otimes x_1 \dots x_n,
$$
or equivalently,
$$
l_n^A=l_n\otimes id_A,
$$
is the induced $\Linf$ algebra structure on $L\otimes \mathfrak{m}_A$. Note that the sum in the formula is finite since the ideal $\mathfrak{m}_A$ is nilpotent.
\end{definition}

%\begin{lemma}\label{MC-nilpotency} Let $L$ be an $L_\infty$ algebra and $\mathfrak{m}$ a commutative $k$-algebra. Then, there exists a natural $L_\infty$-structure on the tensor product $L\otimes_{\mathbb{K}} \mathfrak{m}$. Furthermore, if $\mathfrak{m}$ is nilpotent, then the $L_\infty$-algebra $L\otimes_{\mathbb{K}} \mathfrak{m}$ is nilpotent as well, making the Maurer-Cartan equation in Definition \ref{L infty MC functor} of finite order.\end{lemma}\begin{proof} The structure maps of the tensor product $L\otimes \mathfrak{m}$ are given by:$$l_n(a_1\otimes x_1, \dots, a_n\otimes x_n) = l_n(a_1, \dots, a_n)\otimes x_1 \dots x_n.$$ For the second part, if $\mathfrak{m}$ is nilpotent for a given $n$, the result follows from the relation above.\end{proof} Note that for $A\in \texttt{Art}_{\mathbb{K}}$, its maximal ideal $\mathfrak{m}_A$ is always nilpotent: the Jacobson ideal is $\mathfrak{m}_A$ itself and because of the inherited Artinian condition, we can use Nakayama's. This means that $L\otimes \mathfrak{m}_A$ is a nilpotent $L_\infty$-algebra. Thus, the sum \[\sum_{n\geq1}\frac{l_n(\omega^{\wedge n})}{n!}\] has a finite number of nonzero terms.

%We summarize some important notions about Maurer-Cartan elements.

\begin{definition}\label{MChomotopy} Let $$\bK[t,dt]:=\bK[t]\oplus\bK[t]dt$$ be the cdga of polynomial forms on the affine line $\bA^1_\bK$, that is, with $t$ of degree 0, $dt$ of degree $1$, and the (expected) differential $d(p(t)+q(t)dt)=p'(t)dt$.  For a dga $Z$ we denote by $$Z[t,dt]$$ the dga $Z\otimes\bK[t,dt]$. For any $c\in\bK$ and $z(t,dt)\in Z[t,dt]$, there exists a unique evaluation morphism of dga's $Z[t,dt]\ra Z$, $z(t,dt)\mapsto z(c,0)$, which sends $t$ to $c$ and $dt$ to 0.

For $(A,\mathfrak{m}_A)\in \texttt{Art}$, $A[t,dt]$ is a cdga, finite dimensional over $\bK$. Same for $\mathfrak{m}_A[t,dt]$ which in addition is nilpotent.

The tensor product between an $\Linf$ algebra and a cdga is naturally an $\Linf$ algebra \cite{Ge}*{\S 4}. In the particular case of $L\otimes \mathfrak{m}_A[t,dt]$, the definition of Maurer-Cartan elements becomes

\begin{align*}
MC_L\left(A[t,dt]\right)  :=  \Bigg\{ \omega \in \big(L\otimes \mathfrak{m}_A [t,dt]\big)^1= \left(L^1\otimes\mathfrak{m}_A[t]\right)\oplus \left(L^0\otimes\mathfrak{m}_A[t]dt\right)   & \\
\Bigg| \sum_{n\geq1}\frac{1}{n!} \; l_n^{A[t,dt]}\left(\omega^{\otimes n}\right)= \left(id_L\otimes d_{A[t,dt]}\right)\left(\omega\right) &\Bigg\}.
\end{align*}

The evaluation of any element in this set at $(t,dt)=(c,0)$
 is an element in $\MC_L(A)$ for  any $c\in\bK$. Two elements $\omega_1, \omega_2 \in \MC_{L}(A)$ are said to be {\it homotopy equivalent}, if there exists an element $z(t, dt) \in \MC_{L}(A[t,dt])$ such that $z(0,0)=\omega_1$ and $z(1,0)=\omega_2$.
\end{definition}

Homotopy equivalence is an equivalence relation \cite{Ge}, \cite{Man04}*{Ch. IX}. For a dgla $L$, homotopy equivalence is the same as gauge equivalence; this is an older result of Schlessinger-Stasheff \cite{ScS}, rediscovered by many others, e.g. \cite{FM}*{\S 7}, \cite{Ge}.

\begin{definition} Let $L$ be an $\Linf$ algebra. The {\it deformation functor} of $L$ is the
covariant functor
\begin{equation*}
\Def_L: \texttt{Art} \rightarrow \texttt{Set}
\end{equation*}
defined for all $(A, \mathfrak{m}_A)\in\texttt{Art}$ by
\begin{equation*}
\Def_L: A \mapsto  {\MC_L(A)}_{\big/\sim}
\end{equation*}
where $\sim$ denotes the homotopy equivalence relation between Maurer-Cartan elements.
\end{definition}
%We call the category $\Del(L;A)$ the \textbf{Deligne groupoid} of the $L_\infty$-algebra $L$ at $A$. Its objects are the Maurer-Cartan elements and it makes it a groupoid when associating to them the \emph{action} of the homotopy equivalence. To each $L$, we can associate the \textbf{(Deligne) deformation functor}, a covariant functor \begin{equation*}\Def_L: \texttt{Art}_{\mathbb{K}} \rightarrow \texttt{Set} \end{equation*} defined for all $(A, \mathfrak{m}_A)\in\texttt{Art}_{\mathbb{K}}$ by: \begin{equation*} \Def_L: A \mapsto \Iso \Del(L;A) = \left[\frac{\MC_L(A)}{\text{homotopy equivalences}}\right]. \end{equation*} Where we denote the set of isomorphism classes of a category by $\Iso(-)$, i.e. the right hand side is a set. 

%A morphism of $\Linf$ algebras induces a morphism of the associated deformations functors. Moreover, one has the $\Linf$ algebra version of Theorem \ref{gm88}, see \cite{Man04}*{Corollary IX.22}: 

\begin{theorem}\label{thmWLA}\cite{Man04}*{Corollary IX.22} A weak equivalence of $\Linf$ algebras $L\ra L'$ induces an isomorphism of deformation functors $\Def(L)\ra\Def(L')$. \end{theorem}

\subsection{Cohomology jump functors of $\Linf$ pairs.}

\begin{definition}\label{defCJI}\cite{BW}*{\S 2}
 Let $R$ be a noetherian commutative ring and let $M^\ubul$ be a complex of $R$-modules, bounded above, with finitely generated cohomology. Then, there always exists a bounded above complex $F^\ubul$ of finitely generated free $R$-modules and a quasi-isomorphism of complexes $g: F^\ubul \xrightarrow{\sim} M^\ubul$. The {\it cohomology jump ideals} of $M^\ubul$ are the ideals in $R$ defined as 
$$
J^i_k(M^\ubul) = I_{\text{rk}(F^i)-k+1}(d^{i-1} \oplus d^i),
$$
where $d^i: F^i \rightarrow F^{i+1}$ are the differentials of the complex $F^\ubul$ and $I_r$ is the ideal generated by the size $r$ minors. The cohomology jump ideals do not depend on the choice of the free resolution. 
\end{definition}

\begin{defn}\label{defCJFLP}
Let $(L,M)$ be an $L_\infty$ pair. For all integers $i$ and $k$, the {\it cohomology jump functors} of $(L,M)$ are the functors
$$
\Def^i_k (L,M): \texttt{Art} \rightarrow \texttt{Set}
$$
defined for all $(A,\mathfrak{m}_A) \in \texttt{Art}$ by setting
$$
\Def^i_k (L,M; A) = \bigslant{\{\omega \in \MC_L(A) \hspace{0.1cm}\vert \hspace{0.2cm}J^i_k(M\otimes A, d_\omega) = 0\}}{\sim},
$$
where $\sim$ is the homotopy equivalence relation,  $d_\omega:M\otimes A\ra M\otimes A$ is the $A$-linear differential
$$
d_\omega (\_) :=  \sum_{n\geq 0}\frac{1}{n!} m_{n+1}^A(\omega^{\otimes n}, \_\;)
$$
with $m_n^A=m_n\otimes\id_A$ and $m_n$ are the $\Linf$ $L$-module structure maps on $M$, $0!=1$ as usual, and $J^i_k$ are the cohomology jump ideals of the complex $(M\otimes A, d_\omega)$.
\end{defn}

If $(L,M)$ is a dgl pair, then $\Def^i_k(L,M)$ have been defined and studied in \cite{BW}. The rest of this section is devoted to proving that this definition makes sense for all $\Linf$ pairs:

\begin{theorem}\label{thrmWell} For an $\Linf$ pair $(L,M)$, $\Def^i_k(L,M)$ are well-defined subfunctors of $\Def(L)$. 
\end{theorem}

More precisely, we will show first that $(M\otimes A,d_\omega)$ is indeed a complex of $A$-modules with finitely generated cohomology, so that taking cohomology jump ideals makes sense. Secondly, we will show that moding out by homotopy equivalence makes sense, that is, if $\omega_1$ is a Maurer-Cartan element satisfying the vanishing of the cohomology jump ideal condition, then every other Maurer-Cartan element $\omega_2$ homotopy equivalent to $\omega_1$ will satisfy the same vanishing condition.  We will also prove that a morphism of $\Linf$ pairs induces a morphism of the associated cohomology jump subfunctors, and that the  following $\Linf$  version of \cite{BW}*{Theorem 1.2}  holds:

\begin{theorem}\label{cjdf pairs invariance}
If $(f,g): (L,M)\rightarrow (L',M')$ is a weak equivalence, then for all $i, k\in\bZ$ the induced transformation on subfunctors $$(f,g)_*: \Def^i_k(L,M)\longrightarrow\Def^i_k(L',M')$$ is an isomorphism compatible with the isomorphism of functors $f_*:\Def(L)\ra \Def(L')$ from Theorem \ref{thmWLA}.
\end{theorem}

%\begin{rmk}One can define $i$-quasi-isomorphisms of $\Linf$ pairs extending the notion of $i$-quasi-isomorphisms of dgl pairs. The above theorem has can be then strengthened to say that the isomorphism class of $\Def^i_k(L,M)$ does not change under $i$-quasi-isomorphisms. However, the proof is more involved. We will only use weak equivalences in this article. \end{rmk}

%Together with the Transfer Theorem for dgl pairs, Theorem \ref{thmTTP}, this implies the main result of the deformation theory with cohomology constraints announced in the introduction, Theorem \ref{thmMainLDef}.

\subsection{Twisted complexes.} Before we give the proofs of these theorems, we need some preliminary results.

\begin{proposition} \cite{ChuLaz11} \cite{Yal16} 
\label{twisted L_infty is L_infty} Let $L$ be an $\Linf$ algebra. Let $(A,\mathfrak{m}_A)\in \art$ and consider $\omega\in \MC_L(A)$, a Maurer-Cartan element. Then, one can construct an  $L_\infty$ algebra $(L\otimes A)^\omega$ whose underlying graded vector space is the same as $L\otimes A$, and the  $L_\infty$ structure is
$$
\{l_n^\omega: (L\otimes A)^{\otimes n} \rightarrow L\otimes A\}_{n\geq 1}
$$
defined by
$$
l_n^\omega (a_1,\dots, a_n) = \sum_{i\geq 0} \frac{1}{i!}l_{i+n}(\omega^{\otimes i}, a_1, \dots, a_n).
$$
\end{proposition}

This process is known as {\it twisting by $\omega$}. Note that the sum above is finite, since $\omega\in L^1\otimes\mA$ and $\mA$ is a nilpotent ideal. 

\begin{rmk}
One can show easily that a morphism of $\Linf$ algebras $f:L\ra L'$ induces a morphism of twisted $L_\infty$ algebras $(L\otimes A)^\omega\ra(L'\otimes A)^{f_1(\omega)}$, since for $\omega\in\MC_L(A)$ one has that $f_1(\omega)\in\MC_{L'}(A)$.
\end{rmk}

We introduce twisting of modules via Theorem \ref{L oplus M structure}:

\begin{lemma}
\label{twisted L oplus M} Let $L$ be an $\Linf$ algebra, $M$ an  $L$-module, $A\in \art$, and $\omega\in \MC_{L}(A)$. Then $(\omega, 0) \in \MC_{L\oplus M}(A)$ and the twist by $(\omega, 0)$ of the graded vector space $(L\oplus M)\otimes A$ is the  $L_\infty$ algebra 
%$((L\oplus M)\otimes A)^{(\omega,0)}$
with structure maps
$$\{j^{(\omega,0)}_n: ((L\oplus M)\otimes A)^{\otimes n}\rightarrow (L\oplus M)\otimes A\}_{n\geq 1}
$$
defined by
$$
j^{(\omega,0)}_n\big((a_1, \xi_1), \dots, (a_n, \xi_n)\big) = \sum_{i\geq 0} \frac{1}{i!} j_{i+k}\big((\omega,0)^{\otimes i}, (a_1, \xi_1), \dots (a_n, \xi_n)\big),
$$
where $a_i \in L\otimes A$, $\xi_i \in M\otimes A$, and $j_n$ are as in Theorem \ref{L oplus M structure}.
\end{lemma}

\begin{proof}
First one needs to show that $(\omega, 0)$ is indeed a Maurer-Cartan element in $(L\oplus M)\otimes A$. We can just plug $(\omega, 0)$ in the Maurer-Cartan equation and check that it is equal to zero through Theorem \ref{L oplus M structure}:
$$
\sum_{n\geq 1} \frac{1}{n!} j_n\big((\omega,0)^{\otimes n}) =$$
$$
= \sum_{n\geq 1} \frac{1}{n!} \left(l^A_n(\omega, \dots, \omega), \sum_{i=1}^{n} (-1)^{n-i} m^A_n(\omega, \dots, \hat{\omega}, \dots, \omega, 0)\right)=(0,0).
$$
Here, the first component is zero since $\omega\in \MC_L(A)$. This tells us that we can twist by $(\omega, 0)$. Hence, by Proposition \ref{twisted L_infty is L_infty}, $((L\oplus M)\otimes A)^{(\omega,0)}$ is an $L_\infty$ algebra.
\end{proof}

\begin{lemma}
\label{aomoto definedness} With the notation as in Definition \ref{defCJFLP}, the map $d_\omega$ makes $(M\otimes A, d_\omega)$ into a well-defined complex of $A$-modules.
\end{lemma}
\begin{proof} First of all, $d_\omega$ is a well-defined $A$-linear map since the sum in its definition is finite.  Now, we consider the map $j_1^{(\omega, 0)}$ from  Lemma \ref{twisted L oplus M} and evaluate it at $(0,\xi)\in (L\otimes A)\oplus (M\otimes A)$:
$$
j^{(\omega,0)}_1(0, \xi) = \sum_{i\geq 0} \frac{1}{i!} j_{i+1}\big((\omega,0)^{\otimes i}, (0, \xi)\big) = \sum_{i\geq 0} \frac{1}{i!} \big( l_{i+1}(\omega^{\otimes i}, 0),  m_{i+1}(\omega^{\otimes i-1},0, 0) +\dots
$$
$$
\hspace{5cm}+ m_{i+1}(\omega^{\otimes i-1},0, 0) + m_{i+1}(\omega^{\otimes i}, \xi)\big).
$$
Note that in the second component, when omitting one of the $\omega$'s we get
$$
m_{i+1}(\omega^{\otimes i-1},0, 0)\big) = 0.
$$
So the only nonzero summand is the $(i+1)$-st, omitting $0\in L$. For the first component of $j^\omega_1(0, \xi)$, we always get zero. So then we have 
$$j^{(\omega,0)}_1(0, \xi) = \big(0, \sum_{i\geq 1} \frac{1}{i!}m_{i+1}(\omega^{\otimes i}, \xi)\big)
$$
Restricting the $L_\infty$ algebra differential $j^{(\omega,0)}_1$ to $M\otimes A$, gives us a differential on $M\otimes A$, easily seen to equal $d_\omega$.
\end{proof}

Proposition \ref{L oplus M structure} and Lemma \ref{twisted L oplus M} imply the existence of a natural twisted $\Linf$ module structure $(M\otimes A)^\omega$ on $M\otimes A$ over the twisted $\Linf$ algebra $(L\otimes A)^\omega$:

\begin{proposition}
\label{propMTw}
Let $L$ be an $\Linf$ algebra, $M$ an $\Linf$ $L$-module, $A\in \art$, and $\omega\in \MC_{L}(A)$. Then the collection of restrictions ${j_n^{(\omega, 0)}}_{| M\otimes A}$ is an $\Linf$ $(L\otimes A)^\omega$-module structure on $M\otimes A$, denoted $(M\otimes A)^\omega$. Moreover, there is an equality of twisted $\Linf$ algebras
$$
((L\otimes A)\oplus (M\otimes A))^{(\omega,0)} = (L\otimes A)^\omega\oplus (M\otimes A)^\omega.
$$
\end{proposition}
\begin{proof}
From Lemma \ref{twisted L oplus M}, we know the form of the twisted $\Linf$ algebra structure on the left hand side:
\[j^{(\omega,0)}_n\big((a_1, \xi_1), \dots, (a_n, \xi_n)\big) = \sum_{i\geq 0} \frac{1}{i!} j_{i+n}\big((\omega,0)^{\otimes i}, (a_1, \xi_1), \dots (a_n, \xi_n)\big),\]
where $a_i \in L\otimes A$, $\xi_i \in M\otimes A$, and $j_n$ are as in Theorem \ref{L oplus M structure}.

On the right hand side, going by parts we first have the twisted $\Linf$ algebra structure on $(L\otimes A)^{\omega}$, which by Proposition \ref{twisted L_infty is L_infty} is:
\[l_n^\omega (a_1,\dots, a_n) = \sum_{i\geq 0} \frac{1}{i!}l_{i+n}(\omega^{\otimes i}, a_1, \dots, a_n).\]
The twisted $(L\otimes A)^{\omega}$-module structure is given by ${j^{(\omega,0)}_n}_{\vert M}$. Looking at Proposition \ref{module structure recovered}, we have the following $L^\omega$-module structure:

\[m_n^\omega(a_1,\dots, a_{n-1}, \xi) = \sum_{i\geq 0}\frac{1}{i!}m_{i+n}(\omega^{\otimes i}, a_1, \dots, a_{n-1}, \xi).\]
Now, we can apply Theorem \ref{L oplus M structure} on the structures $l^\omega$ and $m^\omega$, to get
\begin{displaymath}
\begin{split}
j_n^\omega & \big((a_1, \xi_1), \dots, (a_n, \xi_n)\big)\\
&=\left(l^\omega_n(a_1, \dots, a_n), \sum_{q=1}^{n} (-1)^{\Xi(n,q)} m_n^\omega(a_1, \dots, \hat{a}_q, \dots, a_n, \xi_q)\right)\\
&= \left(\sum_{i\geq 0} \frac{1}{i!}l_{i+n}(\omega^{\otimes i}, a_1, \dots, a_n), \sum_{q=1}^{n} (-1)^{\Xi(n,q)} \sum_{i\geq 0}\frac{1}{i!}m_{i+n}(\omega^{\otimes i}, a_1, \dots,\hat{a}_q, \dots, a_{n}, \xi_q)\right)\\
&=  \sum_{i\geq 0}\frac{1}{i!}\left(l_{i+n}(\omega^{\otimes i}, a_1, \dots, a_n), \sum_{q=1}^{n} (-1)^{\Xi(n,q)} m_{i+n}(\omega^{\otimes i}, a_1, \dots,\hat{a}_q, \dots, a_{n}, \xi_q)\right)\\
&= \sum_{i\geq 0} \frac{1}{i!} j_{i+n}\big((\omega,0)^{\otimes i}, (a_1, \xi_1), \dots (a_n, \xi_n)\big)
\end{split}
\end{displaymath}

Thus, we have shown that 
\[j_n^{(\omega,0)}=j_n^\omega.\]

Since the underlying graded vector spaces also coincide, we conclude that the $\Linf$ algebras are the same.
\end{proof}

We will also need the following:

\begin{proposition}
\label{ya16} \cite{Laz}*{\S 6} \cite{Yal16}*{Proposition 3.8} Let $L, L'$ be $\Linf$ algebras,  $A\in\art$, and $\omega\in \MC_{L}(A)$. Let $f: L\otimes A\rightarrow L'\otimes A$ be a weak equivalence of $\Linf$ algebras. Then $f$ gives a weak equivalence of twisted $L_\infty$ algebras $f^\omega: (L\otimes A)^\omega \rightarrow (L'\otimes A)^{f_1(\omega)}$.
\end{proposition}

As a consequence, one has:

\begin{lemma} \cite{Laz}*{\S 6}
\label{twisted qiso} Let $L$ be an $\Linf$ algebra and  $A\in \art$. Then two homotopy equivalent Maurer-Cartan elements $\omega_1 \sim \omega_2$ in $MC_L(A)$ give rise to a weak equivalence of twisted $L_\infty$ algebras
$
(L\otimes A)^{\omega_1} \xrightarrow{\sim} (L\otimes A)^{\omega_2}.
$
\end{lemma}

\begin{cor}
\label{MC-Aomoto}
Let $L$ be an $\Linf$ algebra, $M$ an $\Linf$ $L$-module, and  $A\in \art$. Then two homotopy equivalent Maurer-Cartan elements $\omega_1 \sim \omega_2$ in $MC_L(A)$ give rise to a quasi-isomorphism of twisted complexes $(M\otimes A, d_{\omega_1}) \xra{\sim} (M\otimes A, d_{\omega_2})$.
\end{cor}
\begin{proof} 
We consider as before the $L_\infty$ algebras $L\oplus M$ and $(L\oplus M)\otimes A=(L\otimes A)\oplus (M\otimes A)$ corresponding to $M$ and, respectively, $M\otimes A$. Then, since $\omega_1 \sim \omega_2$, we have $(\omega_1, 0)\sim(\omega_2, 0)$ in $MC_{L\oplus M}(A)$. We apply Lemma \ref{twisted qiso} to get a weak equivalence of twisted $L_\infty$-algebras
$$
\hat\varphi: ((L\otimes A)\oplus (M\otimes A))^{(\omega_1, 0)} \xrightarrow{\sim} ((L\otimes A)\oplus (M\otimes A))^{(\omega_2, 0)}.
$$
That is, $\hat{\varphi}_1$ is a quasi-isomorphism. Using  the last part of Proposition \ref{propMTw}, the restriction of $\hat{\varphi}_1$  to $(M\otimes A)^\omega$ must also be a quasi-isomorphism of complexes. By the proof of Lemma \ref{aomoto definedness}, this is a quasi-isomorphism between the desired twisted complexes.
\end{proof}

\begin{proof}[Proof of Theorem \ref{thrmWell}] Let $\omega\in MC_L(A)$. Then $(M\otimes A,d_\omega)$ is a complex of $A$-modules by Lemma \ref{aomoto definedness}. It is bounded above, since $M$ is  by our convention in this article. 

Moreover, the complex $(M\otimes A,d_\omega)$ has finitely generated cohomology over $A$. The proof of this claim is exactly the same as for the case of dgl pairs \cite{BW}*{Lemma 3.9}. Namely, $(M\otimes A,d_\omega)$ has a finite decreasing filtration of subcomplexes $(M\otimes \mathfrak{m}_A^s,d_\omega)$. The associated spectral sequence $E_1^{s,t}=H^{s+t}(M\otimes \mathfrak{m}_A^s/\mathfrak{m}_A^{s+1},d_\omega)$ degenerates after finitely many pages and abutes to $H^{s+t}(M\otimes A,d_\omega)$. Since $\omega\in L^1\otimes\mathfrak{m}_A$, one has that on $M\otimes \mathfrak{m}_A^s/\mathfrak{m}_A^{s+1}$, $d_\omega=m_1\otimes \id_A$, where $m_1$ is the differential on $M$. Since $(M,m_1)$ has finitely generated cohomology, the same is true for the terms in the spectral sequence, and hence for  $(M\otimes A,d_\omega)$ as well.

We have shown thus that all the conditions from Definition \ref{defCJI} are met. Hence there are well-defined cohomology jump ideals $J^i_k(M\otimes A,d_\omega)$.

Now let $\omega'\in MC_L(A)$ be homotopy equivalent to $\omega$. By Corollary \ref{MC-Aomoto}, the complexes $(M\otimes A,d_\omega)$ and $(M\otimes A,d_{\omega'})$ are quasi-isomorphic. This implies by \cite{BW}*{Corollary 2.3} that they have the same cohomology jump ideals. Thus taking the equivalence relation by homotopy equivalence in the definition of $\Def ^i_k(L,M)$ is  a well-defined operation.

To prove that $\Def^i_k(L,M)\ra \Def(L)$ is a subfunctor, now we only need  to check that a (local) morphism $f:A\ra A'$ in $\texttt{Art}$ leads to a map of sets $$\Def^i_k(L,M;f):\Def^i_k(L,M;A)\ra\Def^i_k(L,M;A')$$ for all $i$ and $k$, induced by restrictions from the map $$\Def(L;f):\Def(L;A)\ra\Def(L;A'),$$ and compatible with compositions.  Let $\omega \in MC_L(A)$ with $J^i_k(M\otimes A,d_\omega)=0$. Then the image $\omega'$ of $\omega$ under the $\Linf$ algebras morphism $L\otimes A\ra L\otimes A'$ is in $MC_{L}(A')$, this being part of the proof that $\Def(L)$ is a functor.
Moreover, from the definition, there is an equality of  twisted complexes
$$
(M\otimes A', d_{\omega'}) = (M\otimes A, d_\omega)\otimes_A A'.
$$
By \cite{BW}*{Corollary 2.4} there is an equality of ideals of $A'$
$$
J^i_k((M\otimes A, d_\omega)\otimes_A{A'}) = J^i_k(M\otimes A, d_\omega)\cdot A'.
$$
Thus
$
J^i_k(M\otimes A',d_{\omega'}) = 0. 
$
Hence the map $\Def^i_k(L,M;f)$ is well-defined. The compatibility of $\Def^i_k(L,M;f)$ with compositions of morphisms in $\texttt{Art}$ follows from the fact that $\Def(L)$ is a functor.
\end{proof}

\subsection{The proof of Theorem \ref{cjdf pairs invariance}} Similar to the proof above, for any morphism of $\Linf$ pairs $(f,g):(L,M)\ra (L',M')$, one obtains commutative diagram of  natural transformations of functors
$$
\xymatrix{
\Def^i_k(L,M) \ar[d]^{(f,g)_*} \ar@{^{(}->}[r] & \Def(L)  \ar[d]^{f_*}\\
 \Def^i_k(L',M') \ar@{^{(}->}[r] & \Def(L')
}
$$
Since $f$ is a weak equivalence, $f_*$ is an isomorphism of functors. To show that $(f,g)_*$ is also an isomorphism of functors, it is enough to show that the twisted complexes $(M\otimes A, d_\omega)$ and $(M'\otimes A, d_{\omega '})$ are quasi-isomorphic for any  $A\in\texttt{Art}$, $\omega\in MC_L(A)$, and $\omega'=f_1(\omega)$ in $MC_{L'}(A)$, by \cite{BW}*{Corollaries 2.3 and 2.4}.

By Proposition \ref{formal pair iff formal oplus}, $(f,g)$ induces a weak equivalence of $\Linf$ algebras $f\oplus g:L\oplus M\ra L'\oplus M'$. By Lemma \ref{twisted L oplus M}, $(\omega,0)$ is in $MC_{L\oplus M}(A)$, and $(\omega',0)$ is in $MC_{L'\oplus M'}(A)$. Moreover, $(\omega',0)=(f\oplus g)_1(\omega,0)=(f_1(\omega),0)$. By Proposition \ref{ya16}, this gives a weak equivalence of $\Linf$ algebras $((L\oplus M)\otimes A)^{(\omega,0)}\ra ((L'\oplus M')\otimes A)^{(\omega',0)}$. By Proposition \ref{propMTw}, this is the same as a weak equivalence
$$
(L\otimes A)^{\omega}\oplus (M\otimes A)^{\omega}\ra (L'\otimes A)^{\omega'}\oplus (M'\otimes A)^{\omega'}.
$$
Hence this defines a weak equivalence of twisted $\Linf$ pairs 
$$
((L\otimes A)^{\omega}, (M\otimes A)^{\omega})\ra ((L'\otimes A)^{\omega'}, (M'\otimes A)^{\omega'})
$$
by Proposition \ref{formal pair iff formal oplus} again. In particular, by the definition of weak equivalence, the first component of the $\Linf$ modules map
$$
(M\otimes A)^{\omega}\ra (M'\otimes A)^{\omega'}.
$$
is a quasi-isomorphism of complexes. We have seen already in the proof of Lemma that first components are precisely the twisted complexes we want, so we get a quasi-isomorphism
$$
(M\otimes A,d_\omega)\ra (M'\otimes A,d_{\omega'}).
$$\hfill\(\Box\)

\subsection{Proof of Theorem \ref{thmMainLDef}}
This follows immediately from Theorem \ref{cjdf pairs invariance} together with Theorem \ref{thmTTP}.

\subsection{Proof of Theorem \ref{propTDef}}\label{subTg} 
Note that $T\Def(L)=T\Def(HL)$ and $T\Def^i_k(L,M)={T}\Def^i_k(HL,HM)$, where the cohomology pair is endowed with an $\Linf$ pair structure making it $\Linf$ weakly equivalent to $(L,M)$. The $\Linf$ pair structure $\{l_n,m_n\}_{n\ge 1}$ on $(HL,HM)$ satisfies $l_1=0$ and $m_1=0$.

It is well-known that $T\Def(L)=T\Def(HL)=H^1L$. We recall the proof. By definition, $T\Def(HL)=\Def(HL; A)$, where $A=\bK[\epsilon]/\epsilon^2$. This is the set of Maurer-Cartan elements in $HL\otimes \epsilon A$ modulo homotopy equivalence. The Maurer-Cartan condition $\sum_{n\ge 1}l_n^A(\omega^n)/n!=0$ is satisfied by all $\omega\in H^1L\otimes \epsilon A$. 

Now let $z(t,dt)\in MC_{HL}(A[t,dt])$. Then $z=z'+z''$ where $z'\in H^1L\otimes \epsilon A[t]$, $z''\in H^0L\otimes \epsilon A[t]dt$, and the Maurer-Cartan condition simplifies to  $(id_L\otimes d/dt)z'=0$. Hence $z'$ must be constant in $t$ and $z''$ is anything. In particular, if $z$ is a homotopy between $\omega_1$ and $\omega_2$, with $\omega_1,\omega_2\in MC_{L}(A)=H^1L\otimes \epsilon A$, then $\omega_1=\omega_2$. Thus, there is no homotopy equivalence to mod out by, and $T\Def(L)=H^1L\otimes \epsilon A\simeq H^1L$.

Now $T\Def^i_k(HL,HM)$ is $\Def^i_k(HL,HM; A)$, with $A$ as above. Since there is no homotopy equivalence to mod out by, this is the set of all $\omega\in H^1L\otimes \epsilon A$ such that the ideal $J^i_k(HM\otimes A,d_\omega)=0$, where $d_\omega(\_)=\sum_{n\ge 0}m_{n+1}^A(\omega^n\otimes \_)/n!$. Again, $m_1=0$ and multiplying more than one $\omega$ is zero. So, $d_\omega(\_)=m_2^A(\omega\otimes \_)$. By definition, $J^i_k(HM\otimes A,m_2^A(\omega\otimes \_))$ is the ideal in $A$ generated by the minors of size $\dim H^iM -k+1$  of 
\be\label{eqm2}
m_2^A(\omega\otimes \_):(H^{i-1}M\oplus H^iM)\otimes A\ra (H^{i}M\oplus H^{i+1}M)\otimes A.
\ee
Since $\omega=\omega'\otimes \epsilon$ for some $\omega'\in H^1L$, $m_2^A(\omega\otimes \_)=m_2(\omega,\_)\otimes \epsilon$. Thus multiplying any two entries of this matrix is zero. Hence, if $\dim H^iM -k+1>1$, all elements of $\Def^i_k(HL,HM; A)=H^1L\otimes \epsilon A= T\Def(HL)$. On the other hand, if $\dim H^1M-k+1=1$, we see that the condition $J^i_k(HM\otimes A,d_\omega)=0$ is the condition that the linear map $m_2(\omega'\otimes\_)$ is zero. That is equivalent to $\omega'$ being in $\ker \tau_i$, with $$
\tau_i:H^1L\ra \bigoplus_{j=i-1,i}\Hom(H^{j}M,H^{j+1}M)
$$ as in the statement of the theorem. \hfill $\Box$

\begin{rmk}
If $(C,M)$ is a dgl pair, then there is a canonical morphism of dgla's $\chi:C\ra \End (M)$. Manetti \cite{M-a} defined in this case a deformation functor $\Def_\chi$ associated to $\chi$. One can show that $\Def_\chi=\Def^i_{h^i}(C,M)$, where $h_i=\dim H^iM$. The Zariski tangent space $\Def_\chi$ as computed in \cite{M-a} agrees with $T\Def^i_{h_i}(C,M)$ as we have computed it in Theorem \ref{propTDef}.
\end{rmk}

\subsection{Proof of Theorem \ref{thrmGTC}}  Firstly, note that by \cite{BW}*{Corollary 2.5}, $R^i_k(A)$ is defined scheme-theoretically by the ideal $J^i_k(A\otimes \cO,d\otimes id_\cO+\omega_{univ})$, where $\cO$ is the affine coordinate ring of $H^1$, and $\omega_{univ}=\sum_je_j\otimes x_j\in H^1\otimes \cO$ is a universal element, that is $e_j$ form a basis of $H^1$ and $x_j$ form the dual basis. The tangent cone at $0$ of $R^i_k(A)$ is an analytic invariant of the germ at $0$ of $R^i_k(A)$. This germ pro-represents $\Def^i_k(A,A)$ up to gauge equivalence, where $(A,A)$ is the associated dgl pair, by the definition of $\Def^i_k(A,A)$. Note that the dgla structure on $A$ is trivial and $A$ is connected, so the there is no gauge equivalence to mod out by. The dgl module structure of $A$ over itself is equivalent to the cdga structure. By Theorem \ref{thmMainLDef}, we can replace $\Def^i_k(A,A)$ by $\Def^i_k(H,H)_\infty$, where $(H,H)_\infty$ is the associated $\Linf$ pair structure on the cohomology pair. By the connectedness assumption and a similar argument as in \ref{subTg}, there is no homotopy equivalence to mod out by. Hence $\Def^i_k(H,H)_\infty$ is given by cohomology jump ideals $J^i_k(H\otimes \cO,d_{univ})$ in the completion $\widehat{\cO}$ at the origin of $\cO$, where $$d_{univ}(\_)=\sum_{n\ge 0} \frac{1}{n!}(\mu_{n+1}\otimes id_\cO)(\omega_{univ}^{\otimes n}\otimes \_\,),$$  and $\mu_n$ are the $\Linf$ module multiplication maps of $H$ on itself. Thus the tangent cone at $0$ of $R^i_k(A)$ is given by the ideal $I^i_k$ generated by the initial forms of elements of $J^i_k(H\otimes \cO,d_{univ})$. By definition, the generators of $J^i_k(H\otimes \cO,d_{univ})$ are the  minors of size $s^i_k=\dim H^i -k+1$ of the matrix of formal power series defined by $d_{univ}^{i-1}\oplus d_{univ}^{i}$. The summands of the entries corresponding to $(\mu_{n+1}\otimes id_\cO)(\omega_{univ}^{\otimes n}\otimes \_\,)$ are homogeneous of degree $n$. If we linearize all the entries of $d_{univ}$, that is, if we replace $d_{univ}(\_)$ by $\mu_{2,univ}(\_):=(\mu_2\otimes id_\cO)(\omega_{univ}\otimes \_\,)$, each minor of size  $s^i_k$ is an initial form (of degree $s^i_k$) of a generator of $J^i_k(H\otimes \cO,d_{univ})$. Hence $J^i_k(H\otimes \cO,\mu_{2,univ})$ is contained in $I^i_k$. However, $R^i_k(H)$ is scheme-theoretically defined by $J^i_k(H\otimes \cO,\mu_{2,univ})$, by \cite{BW}*{Corollary 2.5}. Hence $R^i_k(H)$ contains the tangent cone at $0$ of $R^i_k(A)$.\hfill$\Box$

\section{Cohomology jump loci of rank one local systems}\label{sec4}

In this section we apply the theory of $\Linf$ pairs  to local systems of rank one and prove Theorems \ref{corW}, \ref{thrmLink}, \ref{thrmMFiber}, and \ref{thrmLoc}. Throughout this section the field $\bK$ is $\bC$.

\subsection{} Let $X$ be a connected topological space which has the homotopy type of a finite CW-complex. We consider as in the introduction the complex affine torus $\mb (X)=\Hom (\pi_1(X),\bC^*)$, the moduli space of rank one $\bC$-local systems on $X$, and the cohomology jump loci
$$
\Sigma^i_k(X)=\{\sL\in \mb (X)\mid \dim_\bC H^i(X,\sL)\ge k\}.
$$

Let $\Omega_{{\rm DR}}(X)$ be Sullivan's cdga of piecewise smooth $\bC$-forms on  $X$.  $\Omega_{{\rm DR}}(X)$ can be replaced by the de Rham complex of smooth $\bC$-forms on $X$ if $X$ is a  manifold. For a local system $\sL$, denote by $\Omega_{{\rm DR}}(\sL)$ the module over $\Omega_{{\rm DR}}(X)$  obtained from the forms with values in  $\sL$. 

By Kadeishvili's theorem, there is an $A_\infty$ algebra structure on $H(X,\bC)$ making it $A_\infty$ quasi-isomorphic to $\Omega_{{\rm DR}}(X)$. As we have seen, this extends to modules and passes to  $L_\infty$ pairs. That is, for a local system $\sL$ there is an $L_\infty$ pair structure on $$(H(X,\bC),H(X,\sL))$$ making it weakly $L_\infty$ equivalent to the dgl pair \[(\Omega_{{\rm DR}}(X),\Omega_{{\rm DR}}(\sL)).\]
Since $\Omega_{{\rm DR}}(X)$ is a cdga, itself as a dgla has zero Lie bracket. Hence the $\Linf$ algebra structure on $H(X,\bC)$ is trivial, that is, all products are zero, by Corollary \ref{minimal model of dgla}. The dgl module structure of $\Omega_{{\rm DR}}(X)$ over itself as a dgla, is however nontrivial. Indeed, the module products are the associative products, that is, the wedging of forms in this case. Thus, the only non-trivial part of the $\Linf$ pair structure on $(H(X,\bC),H(X,\sL))$ is the $\Linf$ module multiplication map $m=(m_n)_{n\ge 2}$. We will call this multiplication map $m=(m_n)_{n\ge 2}$ a {\it canonical $\Linf$ module multiplication map}. It depends only on the choices $(f,g,h)$ from homotopy transfer diagrams, as in the previous section. If $\sL$ is the constant sheaf, one has that a canonical $\Linf$ module multiplication map of $H(X,\bC)$ over itself is the same data as an $A_\infty$ algebra structure on $H(X,\bC)$, making it $A_\infty$ equivalent to $\Omega_{{\rm DR}}(X)$. Hence, part (1) of Theorem \ref{thrmLoc} follows from part (2), which we now prove:

\subsection{Proof of Theorem \ref{thrmLoc}} We focus on (2) since it implies (1). By \cite{BW}, the formal germ of $\Sigma^i_k(X)$ at $\sL$  pro-represents the cohomology jump deformation functor $\Def^i_k(\Omega_{{\rm DR}}(X),\Omega_{{\rm DR}}(\sL))$ of the dgl pair. Hence, by Theorem \ref{thmMainLDef}, the formal germ of $\Sigma^i_k(X)$ at $\sL$  pro-represents the cohomology jump deformation functor $$\Def^i_k(H(X,\bC),H(X,\sL))$$ of the $\Linf$ pair $(H(X,\bC),H(X,\sL))$. 

Note first that the deformation functor $\Def(H(X,\bC))$ of the $\Linf$ algebra $H(X,\bC)$ is pro-represented by the formal germ $H^1(X,\bC)_{(\mathbf{0})}$ at the origin of the affine space $H^1(X,\bC)$. This is nothing new, and it can be seen directly from the fact that $H^1(X,\bC)$ is the Lie algebra of the algebraic group $\mb(X)$, whose connected component through the identity is $(\bC^*)^b$, with $b=\dim H^1(X,\bC)$. However, we recall here the proof via $\Linf$ algebras.  Indeed, for $(A,\mathfrak{m}_A)\in \tart$, 
$\Def(H(X,\bC);A)$ is the set  of Maurer-Cartan elements 
$$MC_{H(X,\bC)}(A)=\left\{ \omega \in H^1(X,\bC)\otimes\mathfrak{m}_A \hspace{0.1cm}\middle|\hspace{0.1cm} \sum_{n\ge 2} \frac{1}{n!}\hspace{0.1cm}\mu_n^A (\omega^{\otimes n}) = 0 \right\}$$
modulo homotopy equivalence. Here $\mu_n^A=\mu_n\otimes \id_A$, and $\mu=(\mu_n)_n$ is the $\Linf$ algebra structure  on $H(X,\bC)$. As mentioned above $\mu_n=0$ for all $n$, since $\Omega_{{\rm DR}}(X)$ is a cdga. So
$$
MC_{H(X,\bC)}(A)=H^1(X,\bC)\otimes\mathfrak{m}_A.
$$
Moreover, in this set no two elements are homotopy equivalent. Indeed, for an element $z=(z',z'')$ in $(H^1(X)\otimes \mathfrak{m}_A[t])\oplus (H^0(X)\otimes \mathfrak[t]dt)$, the Maurer-Cartan condition simplifies to $(d/dt)(z')=0$. Hence $z'$ must be constant in $t$. Thus making $dt=0$ and $t=0,1$ in $z$, one obtains the same element in $H^1(X)\otimes \mathfrak{m}_A$.
Hence $\Def(H(X,\bC))=H^1(X,\bC)_{(\mathbf{0})}$ as functors.

Next consider the {\it $\Linf$ resonance} subset of $H^1(X,\bC)$ defined as
$$
\mathcal{R}^i_k (X,\sL) := \left\{ \omega \in H^1(X,\bC) \hspace{0.1cm}\middle|\hspace{0.1cm} \dim H^i(H^\ubul(X,\sL), d_\omega) \geq k\right\}, 
$$
where 
$$
d_\omega (\eta) :=  \sum_{n\geq 0}\frac{1}{n!} m_{n+1}(\omega^{\otimes n}\otimes\eta).
$$
Note that by assumption, $d_\omega(\eta)$ is a finite sum of at most $n_0$ non-zero summands. Since for all $\omega\in H^1(X,\bC)$ one has that $\mu_n(\omega^{\otimes n})=0$ for all $n>0$, $(H^\ubul(X,\sL), d_\omega)$ must be a complex. Indeed, this is classically known, see for example the full version of Proposition \ref{twisted L_infty is L_infty} in \cite{ChuLaz11} and \cite{Yal16}; or one can derive it directly from the compatibility of the $\mu_n$ with the $m_n$.  

One  endows $\mathcal{R}^i_k (X,\sL)$ with a natural  structure of subscheme of $H^1(X,\bC)$ as follows. Let $R$ be the affine coordinate ring of $H^1(X,\bC)$. If $e_1,\ldots ,e_b$ is a basis for the vector space $H^1(X,\bC)$, then  $R=\bC[x_1,\ldots, x_b]$ where $x_1,\ldots,x_b$ is the dual basis in the dual vector space $H^1(X,\bC)^\vee$. Define the universal element
$$
\omega^{univ}:=\sum_{1\le j\le b}e_i\otimes x_i \in H^1(X,\bC)\otimes R.
$$
For $\eta\in H^\ubul(X,\sL)\otimes R$, define an $R$-linear map
$$
d_{univ}(\eta):=\sum_{n\geq 0}\frac{1}{n!} (m_{n+1}\otimes\id_R)(\omega_{univ} ^{\otimes n},\eta).
$$
Again this is a finite sum with at most $n_0$ non-zero summands. Thus $$(H^\ubul(X,\sL)\otimes R, d_{univ})$$ is a well-defined complex interpolating all complexes $(H^\ubul(X,\sL), d_\omega)$. By construction,  $\mathcal{R}^i_k (X,\sL)$ is the zero locus in $\spec (R)$ of the cohomology jump ideal  $J^i_k((H^\ubul(X,\sL)\otimes R, d_{univ}))\in R$ of the universal twisted complex. This defines the scheme structure on $\mathcal{R}^i_k (X,\sL)$.

Moreover, from the definition of the  cohomology jump subfunctors  and the fact that there is no homotopy equivalence to mod out by, one has that the formal germ at the origin of $\mathcal{R}^i_k (X,\sL)$ pro-represents $\Def^i_k(H(X,\bC),H(X,\sL))$.

To summarize, we have obtained the following commutative diagram with vertical arrows isomorphisms of formal germs:
$$
\xymatrix{
\mathcal{R}^i_k (X,\sL)_{(\mathbf{0})} \ar@{^{(}->}[r]  \ar[d]^*[@]{\sim}& H^1(X,\bC)_{(\mathbf{0})} \ar[d]^*[@]{\sim} \\
\Sigma^i_k(X)_{(\sL)} \ar@{^{(}->}[r] & \mb(X)_{(\sL)}.
}
$$
Moreover, the right-most isomorphism is induced by the exponential map 
$$\exp_\sL:T_\sL\mb(X)\ra \mb(X)$$ from the tangent space at $\sL$ of the algebraic group $\mb(X)$, which in its turn is induced by translation from the usual exponential map
$$\exp: \bC^b=H^1(X,\bC)\ra (\bC^*)^b$$
for the connected component $(\bC^*)^b$ of $\mb(X)$ containing $\bone$.

From now on, the proof is the same as that of \cite{BW17}*{Theorem 1.3}. Namely, let $V'$ and $W'$ be two irreducible components of $\Sigma^i_k(X)$ and
$\mathcal{R}^i_k (X,\sL)$, respectively, passing through $\sL$ and $0$, respectively, and such that they correspond to each other under the isomorphism from the above diagram. By translation we obtain two subvarieties $V$ and $W$ of $(\bC^*)^b$ and $\bC^b$, respectively, isomorphic as varieties with $V'$ and $W'$, respectively, and such that the exponential map induces an isomorphism between the germs of $V$ and $W$ at $1$ and $0$, respectively. By  Theorem \ref{propAx}, this implies that $V$ is a subtorus. Hence $V'$ is a translated subtorus in $\mb(X)$. 
$\hfill\Box$

\begin{rmk} 
A statement with a similar conclusion under the assumption that the dgl pair $(\Omega_{{\rm DR}}(X),\Omega_{{\rm DR}}(\sL))$ admits a finite dimensional dgl pair model was proved in \cite{BW17}. Although one would like to say that this finite dimensionality assumption implies the assumption of the previous theorem, this is at the moment not clear to us.
\end{rmk}

\subsection{}
We will use the following result of Cirici-Horel \cite{CH}. As kindly pointed out to us by J. Cirici, this is a particular case of Theorem 8.7 in \cite{CH}, cf. Remark 8.8 from \cite{CH}. This was first proven by Morgan \cite{Mo} in the case of smooth varieties and extended by Cirici-Guill\'en \cite{CG} to possibly singular nilpotent varieties. The approach of \cite{CH} allows to remove the nilpotency conditions in the singular case. 

\begin{theorem}
\label{thrmCH} 
Let $X$ be a  complex algebraic variety, possibly reducible. The cdga $\Omega_{{\rm DR}}(X)$ is quasi-isomorphic to a cdga $A$ such that:

i) for every $n$ there is an extra grading $A^n=\bigoplus_p A_p^n$;

ii) the extra grading on $A$ is  compatible with the differential and with the multiplication on $A$, that is, for all integers $p, n, p', n'$,
$$
d(A_p^n)\subset A_p^{n+1}\quad\text{and}\quad A_p^n\cdot A_{p'}^{n'}\subset A_{p+p'}^{n+n'}.
$$

iii) the filtration $W_pA^n:=\bigoplus_{q\leq p} A^n_q$ induces Deligne's weight filtration on $H^n(A)\cong H^n(X,\bC)$.
\end{theorem}

\begin{rmk}
\label{rmkLinkCH} 
The only input from geometry needed in \cite{CH} to prove Theorem \ref{thrmCH} is the existence of a multiplicative mixed Hodge diagram, enhancing the mixed Hodge structures on cohomology to the rational homotopy setting. For varieties this is available by \cite{Mo}, \cite{Ha}, \cite{Na}. For links as defined in the introduction, this has been constructed by  \cite{DH}. For Milnor fibers of germs of holomorphic functions $f:(\bC^n,0)\ra (\bC,0)$, this has been constructed by  \cite{Na}. The multiplicative mixed Hodge diagrams needed in \cite{CH} must also be cohomologically connected, since the existence {of} Sullivan's minimal models is used. Hence Theorem \ref{thrmCH} also holds for connected links and connected Milnor fibers. Since the multiplicative mixed Hodge diagrams for varieties, links, and Milnor fibers, split into direct sums of diagrams for each connected component, Theorem \ref{thrmCH} also hold for each connected component separately. We thank J. Cirici for this remark.
\end{rmk}

Applying Corollary \ref{corKgr} to this result, we conclude:

\begin{theorem}
\label{thrmStrict2} Let $X$ be a complex algebraic variety. Then on $H^\ubul(X,\bC)$ there exists an extra grading inducing the weight filtration, and there exists an $A_\infty$ structure $\mu=(0,\mu_2,\mu_3,\ldots)$ in the canonical homotopy class of such structures such that all  multiplication maps $\mu_n$ 
are compatible with the multigrading given by the extra grading, as in (\ref{eqMulti}).
\end{theorem}

\subsection{Proof of Theorem \ref{corW}.} Let $d$ be the dimension of $X$. For this proof will use the notation 
\[H^n=H^n(X,\bC)\text{ and }H=H^\ubul(X,\bC).\] 
Recall that $H^n=0$ for $n<0$ and $n>2d$, and the weight filtration satisfies 
\be\label{eqWr}
W_kH^n=\left\{
\begin{array}{rl}
0 & \text{ if }k<0,\\
H^n & \text{ if } k>2n.
\end{array}\right.
\ee

Theorem \ref{thrmStrict2} {provides an $A_\infty$ algebra structure $\mu$ on $H$, as well as an extra grading}
\[
H=\bigoplus_kH_k,
\]
which we call {\it weight decomposition}. Furthermore, this structure is compatible with $\mu$ such that
\[
W_kH^n=\bigoplus_{l\le k}H^n_l.
\]
Note that $H^n_k=0$ if $k<0$ or $k\ge 2n$, by (\ref{eqWr}).

{Although the $\Linf$ algebra structure is trivial} on $H$, the $\Linf$ module multiplication maps of $H$ on itself
$$\mu_n:H^{\otimes n-1}\otimes H\ra H$$
are given by the $A_\infty$ algebra multiplication maps, as one can see from Proposition \ref{Ainf to Linf}, Theorem \ref{L oplus M structure}, and Proposition \ref{module structure recovered}. By compatibility, the maps $\mu_n$ restrict to maps
\[
\mu^{i_1,\ldots,i_n}_{p_1,\ldots,p_n}:H^{i_1}_{p_1}\otimes\ldots\otimes H^{i_n}_{p_n}\ra H^{i_1+\ldots +i_n+2-n}_{p_1+\ldots +p_n}.
\] 

Let now $\omega\in H^1$, $m\in H^i$, and $n>2i+2$. Write
$
m=\sum_{k=0}^{2i} m_k
$
for with $m_k$ in $H_k^i$. Similarly, we have
$
\omega=\omega_1 +\omega_2
$
with $\omega_k$ in $H^1_k$. Note that $H^1_0=0$ since we assume that $W_0H^1=0$. Then every non-zero term in the weight decomposition of $\omega^{\otimes n}\otimes m$ has degree at least $n$, which is the degree of $\omega_1^{\otimes n}\otimes m_0$. Therefore the same is true for $\mu_n(\omega^{\otimes n}\otimes m)$, by compatibility of $\mu_n$ with the weight decomposition. However, $\mu_n(\omega^{\otimes n}\otimes m)$ is an element of $H^{i+1}$, and $H^{i+1}_n=0$ since $n>2(i+1)$. Hence, $\mu_n(\omega^{\otimes n}\otimes m)=0$.

In particular, $
\mu_{n+1}(\omega,\ldots,\omega,m)=0$
for  $\omega\in H^1$, $m\in H$, and $n>4\dim(X)+2$. The claim follows  from Theorem \ref{thrmLoc} (1).
$\hfill\Box$

\subsection{Proof of Theorems \ref{thrmLink} and \ref{thrmMFiber}.} By Remark \ref{rmkLinkCH},  Theorem \ref{thrmCH} is true for each connected component of a link $\cL$, or of a  Milnor fiber $F$, in place of $X$. The rest of the proof is the similar to that of Theorem \ref{corW}. \hfill $\Box$

\subsection{The condition $W_0H^1=0$.}

\begin{rmk}\label{rmkW} Let $X$ be a complex algebraic variety. If the singular locus of $X$ is an isolated point, then
$
W_0H^1(X,\bC)=\tilde{H}^0(\Delta(E),\bC),
$
where $\Delta(E)$ is the dual complex of the exceptional divisor in any resolution of singularities of $X$, see \cite{Pe}.  For a compact variety $X$, $W_0H^1(X,\bC)=H^1(\Delta(X^\bullet),\bC)$ where $\Delta(X^\bullet)$ is the dual complex (or nerve) of a simplicial resolution $X^\bullet$ of $X$ \cite{Ara}.

In general, $
W_0H^1(X,\bC)= H^1(X^{an},\bC), 
$
where $X^{an}$ is the Berkovich analytification of $X$, by \cite{Be}. A down-to-earth topological description for equidimensional varieties has been given by \cite{MS}, where it is also shown that  $W_0H^1(X,\bC)$ is the kernel of $H^1(X,\bC)\ra H^1(\tilde{X},\bC)$ given by the normalization map $\tilde{X}\ra X$.

\end{rmk}

\begin{example}
\label{exBK} Let $X\ra Y$ be a surjective morphism of complex varieties which is a fiber bundle with connected fiber $F$. {Suppose} that $H^1(Y,\bC)\ne 0$ has positive weights and the same for $H^0(Y, {\mathcal{H}^1(F)})$ if non-zero, where $\mathcal{H}^1(F)$ is the local system on $Y$ with fibers $H^1(F,\bC)$. Then $H^1(X,\bC)\ne 0$ but $W_0H^1(X,\bC)=0$, by the Leray-Serre spectral sequence, since the differentials in the latter are morphisms of mixed Hodge structures.

One can construct in this way examples of irreducible multi-branch complex varieties $X$ with  $H^1(X,\bC)\ne 0$ and $W_0H^1(X,\bC)=0$:

Let $C$ be the projectivization of the nodal cubic curve $\{y^2=x^2+x^3\}$ in $\bP^2$, or the union of the axes $\{xy=0\}$. Let $\tau_C$ be the involution given by multiplying the $y$ coordinate by $-1$ in the first case, and exchanging the two axes in the second case. Then $H^1(C/\tau_C,\bC)=0$. Let $E$ be an elliptic curve and fix a torsion point $P$ of order 2 on it. Denote by $\tau_E$ the involution $Q\mapsto Q+P$ of $E$. Then $H^1(E/\tau_E,\bC)=\bC^2$ is pure of weight 1 since $E/\tau_E$ is also an elliptic curve. Let $X$ be the quotient of $C\times E$ by the diagonal involution $(\tau_C,\tau_E)$. Then $X\ra E/\tau_E$ is a fiber bundle map with fiber $C$. Moreover, $H^0(E/\tau_E, \mathcal{H}^1(C))=H^1(C,\bC)^{\tau_C}=H^1(C/\tau_C,\bC)=0$, and thus $H^1(X,\bC)=H^1(E/\tau_E,\bC)=\bC^2$ has no weight-zero.

We thank J. Koll\'ar and B. Wang for these examples. See also \cite{MS}*{\S 2}.
\end{example}

%\begin{rmk} (i) It is likely that Theorem \ref{thrmCH} holds for links (deleted neighborhoods) of algebraic varieties. In that case, one obtains with the same proof Theorem \ref{corW} for links: if a link $M$ satisfies that $W_0H^1(M,\bC)=0$, then every component of $\Sigma^i_k(M)$ containing the constant sheaf is a subtorus. (ii) For links, $W_0H^1$ is not a topological invariant, by Steenbrink-Stevens.(iii) \end{rmk}

\subsection{} The following three propositions on the Malcev Lie algebra of fundamental groups do not involve any deformation theory. We include them for completion.

\begin{prop}
\label{rmkADH}
Let $X$ be a connected complex variety with $W_0H^1(X,\bC)=0$.  Then the associated graded Lie algebra $Gr^W\mathfrak{g}_\bC$ with respect to the weight filtration of the complex Malcev Lie algebra of $\pi_1(X,x)$ is isomorphic to the quotient of a free Lie algebra with generators in degrees $-1$ and $-2$ by a Lie ideal generated in degrees $-2, -3,$ and $-4$. If $X$ is in addition projective, then the generators can be chosen only of degree $-1$ and the relations only of degree $-2$. 
\end{prop}

The proof is the same as \cite{ADH}*{Theorem 1.2} which was stated for normal varieties. Checking this involves a chase of the arguments in {\it loc. cit.} back to a series of results of R. Hain. The main point is that \cite{ADH}*{Lemma 4.1} also holds for our setup:

\begin{prop}
\label{propLL} Let $X$ be a connected complex  variety with $W_0H^1(X,\bC)=0$. Then there is a morphism of graded vector spaces 
$$
\delta: Gr^WH_2(X,\bC)\ra \bL(Gr^WH_1(X,\bC))
$$
such that
$$
Gr^W\mathfrak{g}_\bC=\bL(Gr^WH_1(X,\bC))/\delta(Gr^WH_2(X,\bC)),
$$
where $\bL(E)$ denotes the free Lie algebra spanned by the vector space $E$.
\end{prop}

R. Hain informed us that Proposition \ref{propLL} also holds for other spaces  such as links and Milnor fibers, due to a more general result of his on ``multiplicative mixed Hodge complexes'' \cite{Ha-red}*{Theorem 11.6}, see also Remark \ref{rmkLinkCH} above. We thank him for pointing this out. In particular, one obtains the following statement for links:

\begin{prop}
\label{propMoLi} Let $\cL$ be a  link as defined in the introduction, with $Z'=0$. Let $\cL'$ be a connected component of $\cL$ such that $W_0H^1(\cL',\bC)=0$. Assume that Proposition \ref{propLL} holds for $\cL'$ replacing $X$. Then the associated graded Lie algebra $Gr^W\mathfrak{g}_\bC$ with respect to the weight filtration of the complex Malcev Lie algebra of $\pi_1(\cL')$ is isomorphic to the quotient of a free Lie algebra with generators in degrees $-1$ and $-2$ by a Lie ideal generated in degrees $-2$ and $-3$. 
\end{prop}
\begin{proof}
The proof is the same as \cite{ADH}*{Theorem 1.2}, with the only difference that $H^2(\cL',\bC)$ has weights $0$, $1$, $2$, $3$, but not $4$, see \cite{PSt}*{Theorem 6.14}.
\end{proof}

\end{document}